\documentclass[11pt]{amsart}
\usepackage{amsthm,amsmath,amssymb}

\usepackage{geometry}
\usepackage{graphicx,cite,enumitem}
\usepackage{color}

\numberwithin{equation}{section}
\numberwithin{figure}{section}
\theoremstyle{plain}
\newtheorem{thm}{Theorem}[section]
\newtheorem{lem}[thm]{Lemma}

\theoremstyle{remark}
\newtheorem{rmk}[thm]{Remark}

\newcommand{\M}{\operatorname{M}}

\newcommand{\wt}{\operatorname{wt}}

\begin{document}

\title{Tiling generating functions of halved hexagons and quartered hexagons}

\author{Tri Lai}
\address{Department of Mathematics, University of Nebraska -- Lincoln, Lincoln, NE 68588, U.S.A.}
\email{tlai3@unl.edu}
\thanks{This research was supported in part  by Simons Foundation Collaboration Grant (\# 585923).}

\author{Ranjan Rohatgi}
\address{Department of Mathematics and Computer Science, Saint Mary's College, Notre Dame, IN 46556, U.S.A.}
\email{rrohatgi@saintmarys.edu}

\subjclass[2010]{05A15,  05B45}

\keywords{perfect matchings, plane partitions, lozenge tilings, shuffling phenomenon}

\date{\today}

\dedicatory{}

\begin{abstract}
We prove exact product formulas for the tiling generating functions of various halved hexagons and  quartered hexagons with defects on boundary. Our results generalize the previous work of the first author and the work of Ciucu.
\end{abstract}

\maketitle
\section{Introduction}\label{Sec:Intro}

Working on weighted enumerations of tilings often gives more insights than working on unweighted or `plain' counting of tilings, and can be more challenging. It is a fact that most known results in the field are unweighted tiling enumerations; results about tiling generating functions are very rare. In this paper, we provide simple product formulas for weighted enumerations of tilings of different types of \emph{quartered hexagons} and \emph{halved hexagons}.

A `\emph{halved hexagon}' is half of a vertically symmetric hexagon divided by a vertical zigzag cut along the symmetry axis (see Figures \ref{Fig:Onesidehole}(a) and (b) for examples). The study of halved hexagons began with the work of Proctor on certain classes of staircase plane partitions \cite[Corollary 4.1]{Proc}. His result implies an elegant tiling formula for a hexagon with a maximal staircase cut off, which can be viewed as a halved hexagon with a defect. We note that the tilings of a halved hexagon are in bijection with the \emph{transpose-complementary plane partitions}, one of the ten symmetry classes of plane partitions \cite{Stanley2}. We refer the reader to, e.g., \cite{Halfhex1, Halfhex2, Halfhex3, Ciucu2, CK02, LR, shuffling2, Rohatgi} for more discussion about tiling enumerations of halved hexagons. 

A `\emph{quartered hexagon}' is half of a horizontally symmetric halved hexagon divided along its horizontal symmetric axis. These regions have been investigated in several different contexts, see, e.g., \cite{Tri6, Tri7, Tri9, AF, KGV}. Some of these results show that tilings of quartered hexagons have fundamental connections to antisymmetric monotone triangles and classic group characters. 

As MacMahon's classical theorem on boxed plain partitions \cite{Mac} yields beautiful  $q$-enumerations of lozenge tilings of a `\emph{quasi-regular hexagon}', one would expect the existence of nice $q$-enumerations for halved hexagons and quartered hexagons. Here a \emph{quasi-regular hexagon} is a centrally symmetric hexagon with all $120^{\circ}$ angles.  However, all known enumerations of these regions are unweighted ones. In this paper, we provide nice $q$-enumerations for four different families of quartered hexagons and six families of halved hexagons.

A. Brodin, V. Gorin, and E. M. Rains \cite{Borodin} provide a systematic way to define weights for tilings as follows. They first assign an `elliptic weight' to lozenges on the plane, based on some coordinate system. Then the weight of a tiling is the product of weights of its lozenges. We will adapt and specialize Brodin--Gorin--Rains' elliptic weight in this paper.

The $i$-axis of our coordinate runs along a horizontal lattice line. The length of one unit on the $i$-axis is precisely half the width of a vertical lozenge or, equivalently, half the length of an edge of a lozenge. The $j$-axis is perpendicular to the $i$-axis at a lattice vertex (this vertex is the origin of our coordinate system); the length of one unit on the $j$-axis is equal to half the height of a vertical lozenge or, equivalently, $\sqrt{3}/2$ times the length of an edge of a lozenge. Figure \ref{Fig:weight} shows a particular placement and tiling of a hexagon on our coordinate system.
\begin{figure}\centering
\includegraphics[width=8cm]{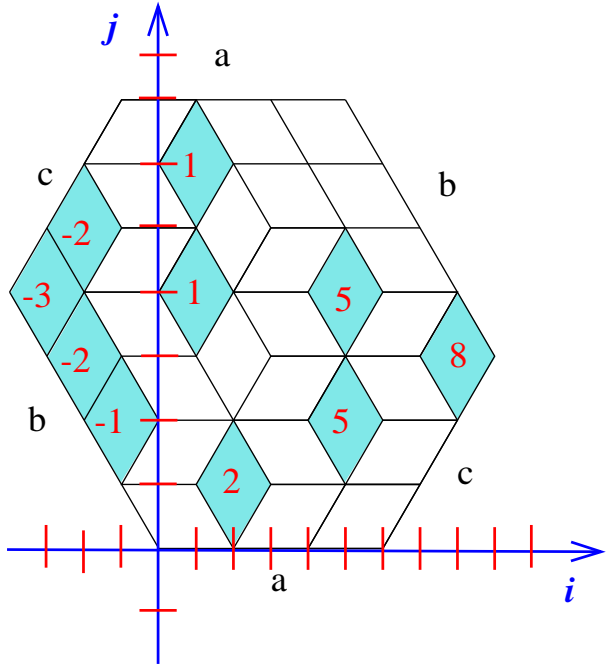}
\caption{Assigning weights to lozenges of a quasi-regular hexagon. Each lozenge with label $x$ has weight $\frac{q^{x}+q^{-x}}{2}$.}\label{Fig:weight}
\end{figure}

Each vertical lozenge with center at the point $(i,j)$ has weight
\begin{equation}
\wt_1(i,j)=\frac{q^{i}+q^{-i}}{2}.
\end{equation}
(The weight does not depend on $j$.) All other lozenges have weight $1$. The weight of a tiling is the product of the weights of its lozenges. We call $\wt_1$ the `\emph{symmetric weight}.' We also consider a variation $\wt_2$ of the weight $\wt_1$ by assigning to each vertical lozenge that intersects the $j$-axis a weight of $1/2$ (these lozenges have weight $1$ in $\wt_1$). 

We note that if we  assign the weight $\wt_3(i,j)=q^{i/2}$ to each vertical lozenge with center $(i,j)$, then each tiling $T$ of the quasi-regular hexagon has weight $C\cdot q^{Vol(T)}$, where $Vol(T)$ is the volume of the plane partition corresponding to the tiling $T$, and $C$ is a constant independent from the choice of tiling $T$. We call $\wt_3$ the `\emph{volume weight}.'  Unfortunately, this weight  does not often give nice $q$-enumerations. In contrast, the weight $\wt_1$ behaves much better. For example, M. Ciucu, T. Eisenk\"{o}lbl, C. Krattenthaler, and D. Zare \cite{CEKZ} have shown that the ``plain" tiling number of a `\emph{cored hexagon}' (a hexagon with an equilateral triangle removed from the center) is always given by a simple product formula. However, there is no such formula for the tiling generating function associated with $\wt_3$. On the other hand, as shown by H. Rosengren \cite{Rosen}, the symmetric weight $\wt_1$ gives a simple product formula for the tiling generating function of the cored hexagon. In this paper, we show that a similar thing happens for the case of halved hexagons and quartered hexagons. The volume weight does not yield a nice $q$-enumeration for them, however the symmetric weight and its variation give us elegant $q$-products.

The remainder of this paper is organized as follows. In Section \ref{Sec:Result}, we state in detail our main results. We provide exact formulas for the tiling generating functions (TGFs) of four families of \emph{quartered hexagons} (see Theorems \ref{quartthm1}--\ref{quartthm4}). Additionally, we investigate six families of halved hexagons with defects on the vertical (west) side. We also provide simple product formulas for their TGFs (see Theorems \ref{typeAthm}--\ref{typeTthm}). Section \ref{Sec:Pre} is devoted to several fundamental results in the enumeration of tilings. We also state two versions of Kuo condensation \cite{Kuo} that will be employed in our proofs. Sections \ref{Sec:Quarter} and \ref{Sec:Twosidehole} contain the proofs of our main theorems. In particular, Theorems \ref{quartthm1}--\ref{quartthm4} are proved in Section \ref{Sec:Quarter}, and Theorems \ref{typeAthm}--\ref{typeTthm} are proved in Section \ref{Sec:Twosidehole}.

\section{Main results}\label{Sec:Result}

\begin{figure}\centering
\setlength{\unitlength}{3947sp}%
\begingroup\makeatletter\ifx\SetFigFont\undefined%
\gdef\SetFigFont#1#2#3#4#5{%
  \reset@font\fontsize{#1}{#2pt}%
  \fontfamily{#3}\fontseries{#4}\fontshape{#5}%
  \selectfont}%
\fi\endgroup%
\resizebox{!}{7cm}{
\begin{picture}(0,0)%
\includegraphics{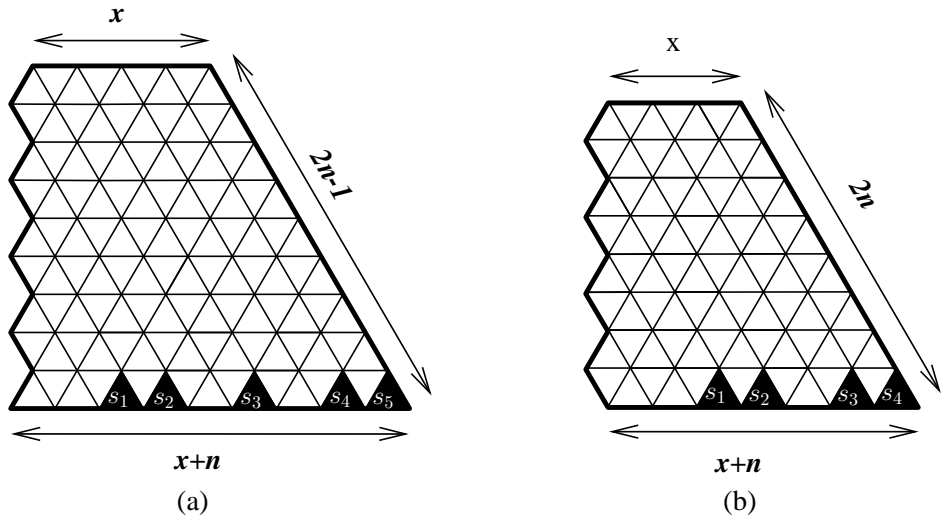}%
\end{picture}%
%
%

\begin{picture}(8654,4913)(2003,-4686)
\put(8454,-3562){\makebox(0,0)[lb]{\smash{{\SetFigFont{14}{16.8}{\rmdefault}{\mddefault}{\itdefault}{\color[rgb]{1,1,1}$s_1$}%
}}}}
\put(9677,-3577){\makebox(0,0)[lb]{\smash{{\SetFigFont{14}{16.8}{\rmdefault}{\mddefault}{\itdefault}{\color[rgb]{1,1,1}$s_3$}%
}}}}
\put(10082,-3569){\makebox(0,0)[lb]{\smash{{\SetFigFont{14}{16.8}{\rmdefault}{\mddefault}{\itdefault}{\color[rgb]{1,1,1}$s_4$}%
}}}}
\put(2941,-3601){\makebox(0,0)[lb]{\smash{{\SetFigFont{14}{16.8}{\rmdefault}{\mddefault}{\itdefault}{\color[rgb]{1,1,1}$s_1$}%
}}}}
\put(3361,-3609){\makebox(0,0)[lb]{\smash{{\SetFigFont{14}{16.8}{\rmdefault}{\mddefault}{\itdefault}{\color[rgb]{1,1,1}$s_2$}%
}}}}
\put(8859,-3577){\makebox(0,0)[lb]{\smash{{\SetFigFont{14}{16.8}{\rmdefault}{\mddefault}{\itdefault}{\color[rgb]{1,1,1}$s_2$}%
}}}}
\put(4164,-3609){\makebox(0,0)[lb]{\smash{{\SetFigFont{14}{16.8}{\rmdefault}{\mddefault}{\itdefault}{\color[rgb]{1,1,1}$s_3$}%
}}}}
\put(4989,-3601){\makebox(0,0)[lb]{\smash{{\SetFigFont{14}{16.8}{\rmdefault}{\mddefault}{\itdefault}{\color[rgb]{1,1,1}$s_4$}%
}}}}
\put(5386,-3609){\makebox(0,0)[lb]{\smash{{\SetFigFont{14}{16.8}{\rmdefault}{\mddefault}{\itdefault}{\color[rgb]{1,1,1}$s_5$}%
}}}}
\end{picture}}
\caption{Two different types of quartered hexagons.}\label{Fig:Quarter}
\end{figure}

\begin{figure}\centering
\includegraphics[width=15cm]{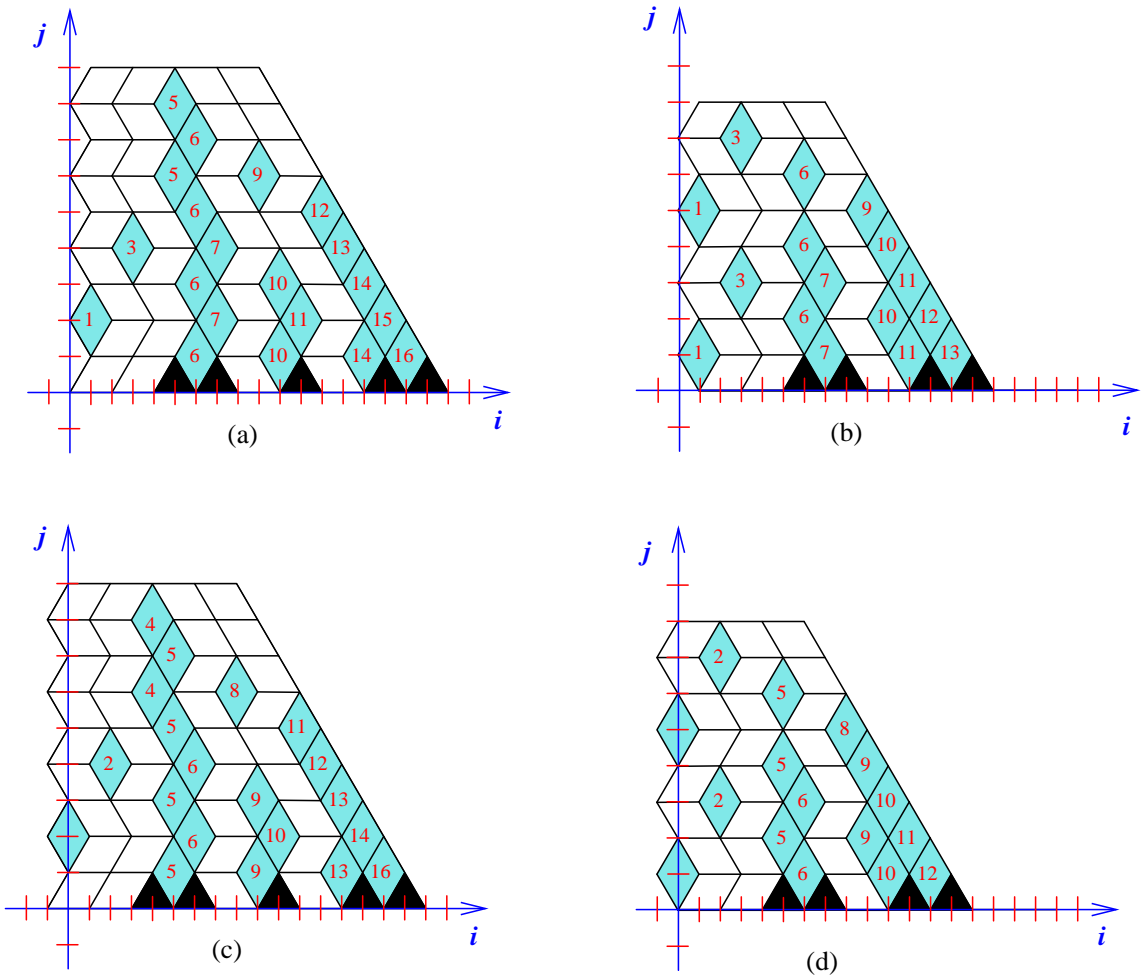}
\caption{Four different weight assignments for four types of weighted quartered hexagons.}\label{Fig:QuarterTile}
\end{figure}

 All regions considered in this paper are weighted regions. Strictly speaking, the a `\emph{weighted region}' is a pair $(R,\wt)$, where $R$ is an unweighted region on the triangular lattice, called the ``shape" of the region, and $\wt$ is a weight assignment for the tilings of $R$. We will see in the next part of the paper that there exist different weighted regions that have the same shape. Whenever the weight assignment is clearly given, we abuse the notation by viewing $R$ as the weighted region. In the rest of the paper, we use the notation $\M(R)$ for the weighted sum of all tilings of $R$. If $R$ does not have any tiling, then $\M(R)=0$. When $R$ is a degenerated region (i.e., a region with empty interior),
$\M(R)=1$ by convention. We call $\M(R)$ the \emph{tiling generating function} of $R$.

\subsection{Quartered hexagons}
While investigating a generalization of  Jockusch--Propp's quartered Aztec diamond \cite{JP}, the first author proved simple product formulas for four families of the quartered hexagons \cite{Tri6, Tri7} (see Figure \ref{Fig:Quarter}). The tilings of the quartered hexagons have interesting connections to antisymmetric monotone triangles \cite{Tri9} and characters of classical groups \cite{AF, KGV}. 

The four different types of quartered hexagons are as follows. 

We start with a right trapezoidal region whose side-lengths are $x,2n-1,x+n,2n-1$ in counter-clockwise order, starting from the north side\footnote{From now on, we always list the side-lengths of a region in this order.}. The north, northeast, and south sides of the region follow lattice lines, and the west side follows a vertical zigzag. We also remove from the base of the region $n$ up-pointing unit triangles at the positions $s_1,s_2,\dots,s_n$, from left to right (see the shape of the region in  Figure \ref{Fig:Quarter}(a); the black unit triangles indicate the removed ones). These removed triangles are called ``\emph{dents}." Next, we assign weight to lozenges of the region (which may be used in tilings) using the symmetric weight $\wt_1$ as in Figure \ref{Fig:QuarterTile}(a).  The $j$-axis is touching the vertical west side of the region, and the $i$-axis runs along the base.  The vertical lozenge with center at the point $(i,j)$ has weight $\frac{q^{i}+q^{-i}}{2}$, and the weight of a tiling is the product of lozenge-weights as usual. Denote by $R^{1}_{x}(s_1,s_2,\dots,s_n)$ the resulting weighted region.

The second family of weighted quartered hexagons is defined similarly. We start with a right trapezoidal region of side-lengths $x,2n,x+n,2n$ and remove $n$ up-pointing unit triangles along the base. Assume that the positions of the removed unit triangles are $s_1,s_2,\dots,s_n$. See Figure \ref{Fig:Quarter}(b) for the shape of the region. We again use the symmetric weight $\wt_1$ for lozenges in the new region, as shown in Figure \ref{Fig:QuarterTile}(b).  Denote by $R^{2}_{x}(s_1,s_2,\dots,s_n)$ this new weighted region. The TGFs of these two types of quartered hexagons are given by simple product formulas.

\begin{thm}\label{quartthm1} Assume that $x$ and $n$ are non-negative integers and $(s_i)_{i=1}^{n}$ is a sequence of positive integers between $1$ and $n+x$. Then
\begin{align}\label{quarterformula1}
\M(R^{1}_{x}(s_1,s_2,\dots,s_n))&=2^{-n(n-1)}q^{-\sum_{i=1}^n(i-1)(4s_i-2i-1)}\frac{\prod_{1\leq i < j\leq n}[2(s_i+s_j-1)]_{q^2}[2(s_j-s_i)]_{q^2}}{\prod_{i=1}^{n}[2i-2]_{q^2}!}\notag\\
&=2^{-n(n-1)}q^{-\sum_{i=1}^n(i-1)(4s_i-2i-1)}\prod_{1\leq i < j\leq n}\frac{[2(s_i+s_j-1)]_{q^2}}{[i+j-1]_{q^2}}\prod_{1\leq i < j\leq n}\frac{[2(s_j-s_i)]_{q^2}}{[j-i]_{q^2}}.
\end{align}
\end{thm}

\begin{thm}\label{quartthm2} Assume that $x$ and $n$ are non-negative integers and $(s_i)_{i=1}^{n}$ is a sequence of positive integers between $1$ and $n+x$. Then
\begin{align}\label{quarterformula2}
\M(R^{2}_{x}&(s_1,s_2,\dots,s_n))=2^{-n^2}q^{-\sum_{i=1}^n((4i-2)s_i-2i^2+i)}\frac{\prod_{i=1}^{n}[2s_i]_{q^2}\prod_{1\leq i < j\leq n}[2(s_i+s_j)]_{q^2}[2(s_j-s_i)]_{q^2}}{\prod_{i=1}^{n}[2i-1]_{q^2}!}\notag\\
&=2^{-n^2}q^{-\sum_{i=1}^n((4i-2)s_i-2i^2+i)}\prod_{i=1}^{n}\frac{1}{q^{4s_i}+1}\prod_{1\leq i \leq j\leq n}\frac{[2(s_i+s_j)]_{q^2}}{[i+j-1]_{q^2}} \prod_{1\leq i < j\leq n}\frac{[2(s_j-s_i)]_{q^2}}{[j-i]_{q^2}}.
\end{align}
\end{thm}

We are also interested in the siblings of the previous regions which use the variation $\wt_2$  of the symmetric weight $\wt_1$. The third family of weighted quartered hexagons has the same shape as the first one; the only difference is the weight assignment $\wt_2$ has been used. In particular, the $j$-axis now passes through the center of the vertical lozenges along the west side of the region. Each vertical lozenge with center at the point $(i,j)$ is still weighted by $\frac{{q}^{i}+q^{-i}}{2}$, with one exception: the  vertical lozenges intersected by the $j$-axis are weighted by $1/2$ (see Figure \ref{Fig:QuarterTile}(c)). Denote by $R^{3}_{x}(s_1,s_2,\dots,s_n)$ this variation of the $R^{1}$-type regions. Finally, the weighted quartered hexagon $R^{4}_{x}(s_1,s_2,\dots,s_n)$ has the same shape as the one in the second family; the only difference is the weight assignment. We again use the weight $\wt_2$, instead of $\wt_1$. See Figure \ref{Fig:QuarterTile}(d) for the details. We also have nice formulas for the TGFs of these new quartered hexagons.

\begin{thm}\label{quartthm3} Assume that $x$ and $n$ are non-negative integers and $(s_i)_{i=1}^{n}$ is a sequence of positive integers between $1$ and $n+x$. Then
\begin{align}\label{quarterformula3}
\M(R^{3}_{x}&(s_1,s_2,\dots,s_n))=2^{-n(n-1)}q^{-\sum_{i=1}^n((4i-4)s_i-2i^2-i+3)}\frac{\prod_{1\leq i < j\leq n}[2(s_i+s_j-2)]_{q^2}[2(s_j-s_i)]_{q^2}}{\prod_{i=1}^{n}[2i-2]_{q^2}!}\notag\\
&=2^{-n(n-1)}q^{-\sum_{i=1}^n((4i-4)s_i-2i^2-i+3)}\prod_{1\leq i < j\leq n}\frac{[2(s_i+s_j-2)]_{q^2}}{[i+j-1]_{q^2}} \prod_{1\leq i < j\leq n}\frac{[2(s_j-s_i)]_{q^2}}{[j-i]_{q^2}}.
\end{align}
\end{thm}

\begin{thm}\label{quartthm4} Assume that $x$ and $n$ are non-negative integers and $(s_i)_{i=1}^{n}$ is a sequence of positive integers between $1$ and $n+x$. Then
\begin{align}\label{quarterformula4}
\M(R^{4}_{x}&(s_1,s_2,\dots,s_n))=2^{-n^2}q^{-\sum_{i=1}^n((4i-2)s_i-2i^2-i+1)}\notag\\
&\times\frac{\prod_{i=1}^{n}[2s_i-1]_{q^2}\prod_{1\leq i < j\leq n}[2(s_i+s_j-1)]_{q^2}[2(s_j-s_i)]_{q^2}}{\prod_{i=1}^{n}[2i-1]_{q^2}!} \notag\\
&=2^{-n^2}q^{-\sum_{i=1}^n((4i-2)s_i-2i^2-i+1)}\prod_{i=1}^{n}\frac{1}{q^{4s_i-2}+1}\prod_{1\leq i \leq j\leq n}\frac{[2(s_i+s_j-1)]_{q^2}}{[i+j-1]_{q^2}} \prod_{1\leq i < j\leq n}\frac{[2(s_j-s_i)]_{q^2}}{[j-i]_{q^2}}.
\end{align}
\end{thm}

\begin{rmk}[Combinatorial reciprocity phenomenon]\label{reciprocity}
From the four formulas (\ref{quarterformula1})--(\ref{quarterformula4}) above, we realize that the  TGF of $R^{3}_{x}(s_1,s_2,\dots,s_n)$ is obtained from the TGF of $R^{1}_{x}(s_1,s_2,\dots,s_n)$ by replacing  $s_i$ by $s_i-1/2$, for $i=1,2,\dots,n$ and similarly, the TGF of $R^{4}_{x}(s_1,s_2,\dots,s_n)$ is obtained from the TGF of $R^{2}_{x}(s_1,s_2,\dots,s_n)$ by making the same replacement. These observations remind us of the ``\emph{combinatorial reciprocity phenomenon}": even though the region $R^{1}_{x}(s_1,s_2,\dots,s_n)$ (resp. $R^{2}_{x}(s_1,s_2,\dots,s_n)$) is \emph{not} defined when the $s_i$'s are half-integers, its tiling formula gives the number of combinatorial objects of a different sort (here, the tilings of $R^{3}_{x}(s_1,s_2,\dots,s_n)$ (resp. $R^{4}_{x}(s_1,s_2,\dots,s_n)$) when evaluated at half-integers. We refer the reader to, e.g., \cite{Beck,StanleyRecip,ProppRecip} for more discussions about the phenomenon. It would be very interesting to have a direct explanation for this which does not require calculating the TGFs.
\end{rmk}

\begin{rmk}
As shown in \cite{AF}, symplectic and orthogonal characters give certain weighted tiling enumerations of  the quartered hexagons. In particular, by evaluating the principal specialization of the classical group characters  one could obtain nice $q$-enumerations of the quartered hexagon. However, these $q$-enumerations are different from those in our Theorems \ref{quarterformula1}--\ref{quarterformula4}, as the weight assignments for the lozenges in the two cases are different. There is no obvious way to go from one to the other. In other words, the work in \cite{AF} and the work in this paper are incomparable.
\end{rmk}

\begin{rmk}
Strictly speaking, our tiling formulas in Theorems  \ref{quarterformula1}--\ref{quarterformula4} are written in terms of $q^2$-integers. If we replace our lozenge weight $\wt_1$ by the weight $\wt'(i,j)=\frac{q^{i/2}+q^{-i/2}}{2}$, we get formulas in terms of $q$-integers. However, we decided to keep the original weight to make our intermediate calculations ``clean." In particular, we do not need to deal with half-integer exponents of $q$.
\end{rmk}

\begin{figure}\centering
\setlength{\unitlength}{3947sp}
\begingroup\makeatletter\ifx\SetFigFont\undefined%
\gdef\SetFigFont#1#2#3#4#5{%
  \reset@font\fontsize{#1}{#2pt}%
  \fontfamily{#3}\fontseries{#4}\fontshape{#5}%
  \selectfont}%
\fi\endgroup%
\resizebox{!}{14cm}{
\begin{picture}(0,0)%
\includegraphics{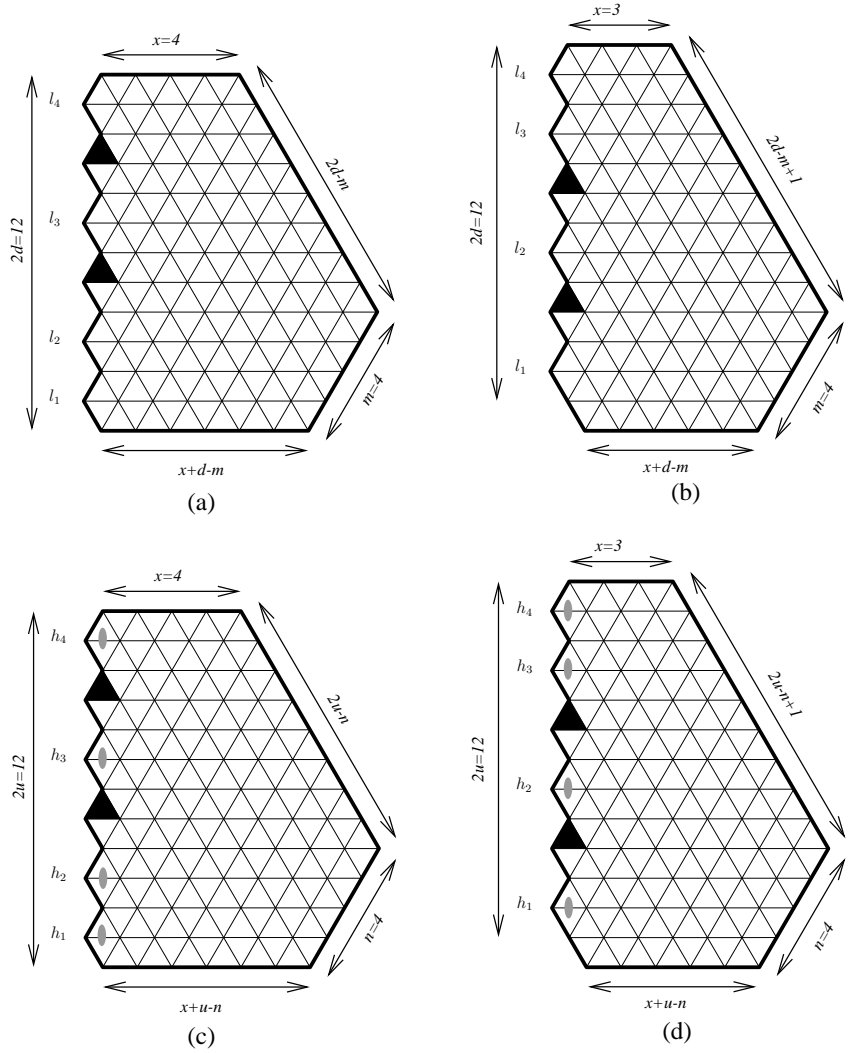}%
\end{picture}%

\begin{picture}(9982,12604)(903,-11782)
\put(1433,-4004){\makebox(0,0)[lb]{\smash{{\SetFigFont{12}{14.4}{\rmdefault}{\mddefault}{\itdefault}{\color[rgb]{0,0,0}$l_1$}%
}}}}
\put(1433,-3295){\makebox(0,0)[lb]{\smash{{\SetFigFont{12}{14.4}{\rmdefault}{\mddefault}{\itdefault}{\color[rgb]{0,0,0}$l_2$}%
}}}}
\put(1433,-1878){\makebox(0,0)[lb]{\smash{{\SetFigFont{12}{14.4}{\rmdefault}{\mddefault}{\itdefault}{\color[rgb]{0,0,0}$l_3$}%
}}}}
\put(1433,-460){\makebox(0,0)[lb]{\smash{{\SetFigFont{12}{14.4}{\rmdefault}{\mddefault}{\itdefault}{\color[rgb]{0,0,0}$l_4$}%
}}}}
\put(6946,-129){\makebox(0,0)[lb]{\smash{{\SetFigFont{12}{14.4}{\rmdefault}{\mddefault}{\itdefault}{\color[rgb]{0,0,0}$l_4$}%
}}}}
\put(6950,-819){\makebox(0,0)[lb]{\smash{{\SetFigFont{12}{14.4}{\rmdefault}{\mddefault}{\itdefault}{\color[rgb]{0,0,0}$l_3$}%
}}}}
\put(6943,-2229){\makebox(0,0)[lb]{\smash{{\SetFigFont{12}{14.4}{\rmdefault}{\mddefault}{\itdefault}{\color[rgb]{0,0,0}$l_2$}%
}}}}
\put(6950,-3647){\makebox(0,0)[lb]{\smash{{\SetFigFont{12}{14.4}{\rmdefault}{\mddefault}{\itdefault}{\color[rgb]{0,0,0}$l_1$}%
}}}}
\put(1453,-10409){\makebox(0,0)[lb]{\smash{{\SetFigFont{12}{14.4}{\rmdefault}{\mddefault}{\itdefault}{\color[rgb]{0,0,0}$h_1$}%
}}}}
\put(1453,-9700){\makebox(0,0)[lb]{\smash{{\SetFigFont{12}{14.4}{\rmdefault}{\mddefault}{\itdefault}{\color[rgb]{0,0,0}$h_2$}%
}}}}
\put(1453,-8283){\makebox(0,0)[lb]{\smash{{\SetFigFont{12}{14.4}{\rmdefault}{\mddefault}{\itdefault}{\color[rgb]{0,0,0}$h_3$}%
}}}}
\put(1453,-6865){\makebox(0,0)[lb]{\smash{{\SetFigFont{12}{14.4}{\rmdefault}{\mddefault}{\itdefault}{\color[rgb]{0,0,0}$h_4$}%
}}}}
\put(6966,-6534){\makebox(0,0)[lb]{\smash{{\SetFigFont{12}{14.4}{\rmdefault}{\mddefault}{\itdefault}{\color[rgb]{0,0,0}$h_4$}%
}}}}
\put(6970,-7224){\makebox(0,0)[lb]{\smash{{\SetFigFont{12}{14.4}{\rmdefault}{\mddefault}{\itdefault}{\color[rgb]{0,0,0}$h_3$}%
}}}}
\put(6963,-8634){\makebox(0,0)[lb]{\smash{{\SetFigFont{12}{14.4}{\rmdefault}{\mddefault}{\itdefault}{\color[rgb]{0,0,0}$h_2$}%
}}}}
\put(6970,-10052){\makebox(0,0)[lb]{\smash{{\SetFigFont{12}{14.4}{\rmdefault}{\mddefault}{\itdefault}{\color[rgb]{0,0,0}$h_1$}%
}}}}
\end{picture}}
\caption{Four regions with holes on the vertical side. The lozenges with shaded cores have weight $1/2$.}\label{Fig:Onesidehole}
\end{figure}

\begin{figure}\centering
\includegraphics[width=13cm]{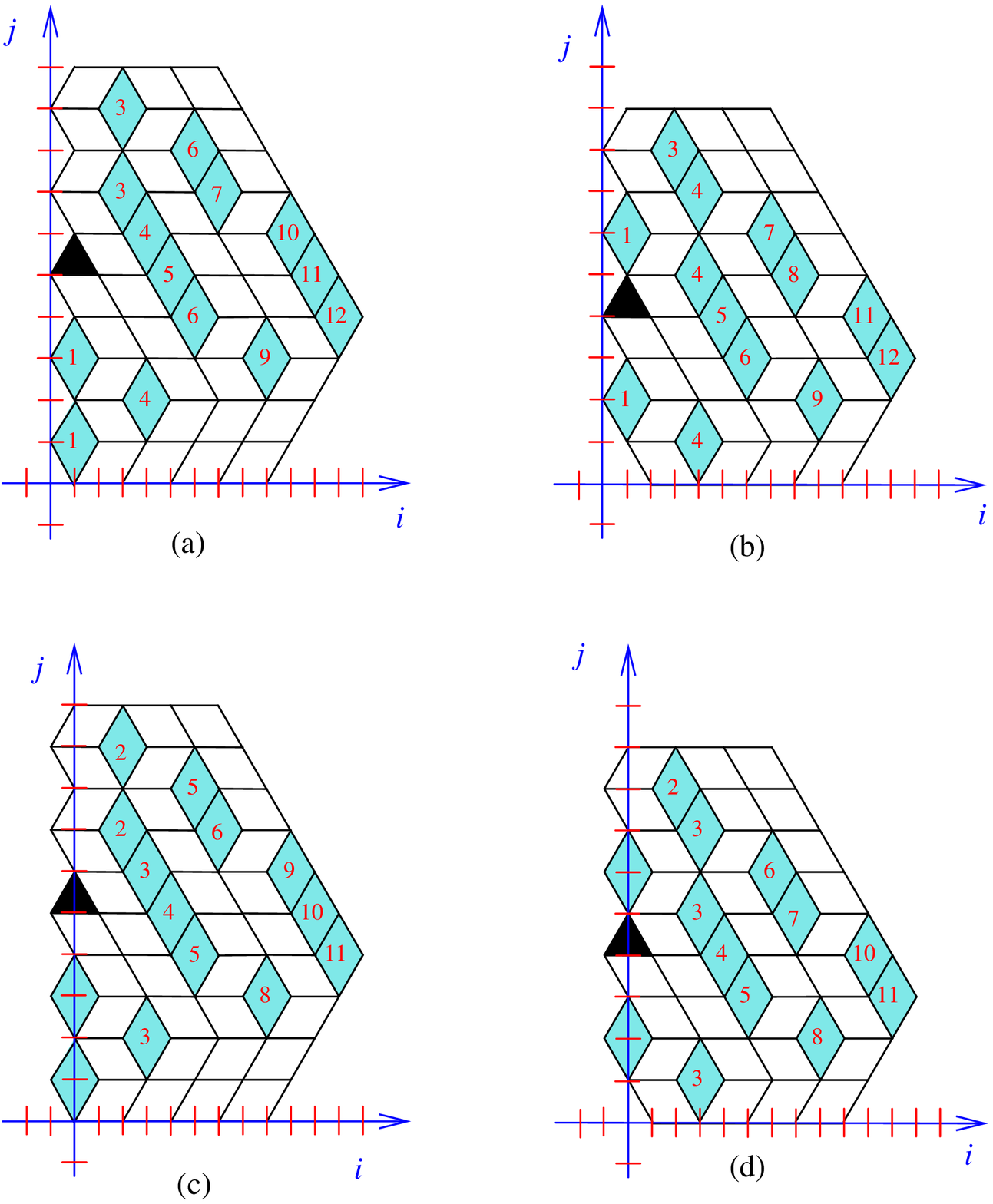}
\caption{Weight assignments for the four families of halved hexagons.}\label{Fig:OnesideholeTile}
\end{figure}

\subsection{Halved hexagons with dents on vertical side}
We consider next four families of halved hexagons with dents on the vertical side. 

The first family is illustrated in Figure  \ref{Fig:Onesidehole}(a). We start with a halved hexagon whose north, northeast, southeast, and south sides have side-lengths  $x,2d-m,m,x+d-m$ ($m\leq d$). The west side runs along a vertical zigzag with $2d$ steps. Along the west side, we remove $d-m$ up-pointing unit triangles (the black ones).  We label  the `bumps' on the west side $1,2,\dots,d$ (from bottom to the top). Assume that the positions of the bumps which do not contain a removed unit triangle are $1\leq l_1<l_2\cdots<l_m\leq d$. We now assign weights to the lozenges of the region using  the weight asignment $\wt_1$, as shown by example in Figure \ref{Fig:OnesideholeTile}(a). Denote by $A_{x,d}(l_1,l_2,\dots,l_m)$ the resulting weighted region.

We next consider three more variations of $A_{x,d}(l_1,\dots,l_m)$ as follows. The second family of weighted regions uses the same weight assignment $\wt_1$ as the $A$-type regions; the differences are the length of the northeast side (which now has side-length $2d-m+1$) and the length of the bottom-most step of the zigzag on the west side (which is now $2$). Denote by $B_{x,d}(l_1,\dots,l_m)$ the resulting region. See Figure \ref{Fig:Onesidehole}(b) for the shape of the $B$-type region and Figure \ref{Fig:OnesideholeTile}(b) for the weight assignment of the region.

The third family of regions has exactly same shape as the $A$-type region, however we now  use the weight asignment $\wt_2$. Recall that the vertical lozenges intersected by the $j$-axis now have weight $1/2$. We denote by $C_{x,u}(h_1,\dots,h_n)$ these regions (shown in Figures \ref{Fig:Onesidehole}(c) and \ref{Fig:OnesideholeTile}(c)).  Similarly, the regions in the fourth family are identical in shape to the $B$-type regions in; the only difference is that we now apply weight assignment $\wt_2$ to their lozenges.  These regions are denoted by $D_{x,u}(h_1,\dots,h_n)$ (illustrated in \ref{Fig:Onesidehole}(d) and \ref{Fig:OnesideholeTile}(d)).

The TGFs of all four regions above are given by simple product formulas.

We adopt the following notation for $0\leq a\leq b$:
\begin{equation}
\label{sq}
\left\langle\begin{matrix}
b\\ a\end{matrix}\right\rangle:=a+(a+1)+(a+2)+\dotsc+b.
\end{equation}

We define two polynomials as follows:
\begin{align}\label{Pdefine}
P(x,u,d,(l_i)_{i=1}^{m}, (h_j)_{j=1}^{n})&=2^{-E}q^{-F}\frac{\prod_{1\leq i < j \leq m}[2(l_j-l_i)]_{q^2}\prod_{1\leq i < j \leq n}[2(h_j-h_i)]_{q^2}}{\prod_{i=1}^{m}\prod_{j=1}^{n}[2(l_i+h_j)]_{q^2}\prod_{i=1}^{m}[2l_i-1]_{q^2}!\prod_{j=1}^{n}[2h_j]_{q^2}!}\notag\\
&\times \prod_{i=1}^{\lceil m/2\rceil}\prod_{j=1}^{2m-4i+3}[2\overline{x}+2i+j-1]_{q^2} \notag\\
&\times \prod_{i=1}^{n}[2(\overline{x}+m+i)]_{q^2} \prod_{i=1}^{n}\prod_{j=1}^{m+i}[2\overline{x}+m+i+j]_{q^2}\notag\\
&\times\prod_{i=1}^{m}\prod_{j=1}^n\frac{[2(\overline{x}+i+j-1)]_{q^2}}{[2(\overline{x}+i+j-1)+1]_{q^2}}\notag\\
&\times \prod_{i=1}^{m}\prod_{j=1}^{l_i-i}[2(\overline{x}+i+j+n)]_{q^2}[2(\overline{x}-i-j+m+1)]_{q^2}\notag\\
&\times \prod_{i=1}^{n}\prod_{j=1}^{h_i-i}[2(\overline{x}+i+j+m)]_{q^2}[2(\overline{x}-i-j+n+1)]_{q^2},
\end{align}
where $\overline{x}=x+u-n+d-m-e=x+\max(u-n,d-m)$ and $E$ and $F$ are defined below.

\begin{align}
Q(x,u,d,(l_i)_{i=1}^{m}, (h_j)_{j=1}^{n})&=2^{-E'}q^{-F'}\frac{\prod_{1\leq i < j \leq m}[2(l_j-l_i)]_{q^2}\prod_{1\leq i < j \leq n}[2(h_j-h_i)]_{q^2}}{\prod_{i=1}^{m}\prod_{j=1}^{n}[2(l_i+h_j)]_{q^2}\prod_{i=1}^{m}[2l_i]_{q^2}!\prod_{j=1}^{n}[2h_j-1]_{q^2}!}\notag\\
&\times \prod_{i=1}^{\lceil n/2\rceil}\prod_{j=1}^{2n-4i+3}[2\tilde{x}+2i+j]_{q^2}\notag\\
&\times \prod_{i=1}^{m}[2(\tilde{x}+n+i)]_{q^2} [2(\tilde{x}+n+i+1)]_{q^2}\prod_{i=1}^{m}\prod_{j=1}^{n+i-1}[2\tilde{x}+n+i+j+1]_{q^2}\notag\\
&\times\prod_{i=1}^{n}\prod_{j=1}^m\frac{[2(\tilde{x}+i+j-1)]_{q^2}}{[2(\tilde{x}+i+j-1)+1]_{q^2}}\notag\\
&\times \prod_{i=1}^{m}\prod_{j=1}^{l_i-i}[2(\tilde{x}+i+j+n+1)]_{q^2}[2(\tilde{x}-i-j+m+1)]_{q^2}\notag\\
&\times \prod_{i=1}^{n}\prod_{j=1}^{h_i-i}[2(\tilde{x}+i+j+m)]_{q^2}[2(\tilde{x}-i-j+n+2)]_{q^2}.
\end{align}
where $\tilde{x}=x+u-n+d-m-e'=x+\max(u-n,d-m+1)-1$ and $E'$ and $F'$ are defined below.

We define the exponents $E,F,E',F'$ as follows:

\begin{align}\label{Eexp1}
E=E(x,u,d,\textbf{l},\textbf{h})&=f(t_1+2d-m,n,m+1)\notag\\
&+\sum_{i=1}^{m}\left\langle\begin{matrix}
t_1+2d-m-(m-i+1)\\ t_1+2(d-l_i)-2(m-i) \end{matrix}\right\rangle \notag\\
&+\sum_{i=1}^{n}\left\langle\begin{matrix}
t_2+2u-n-(n-i+1) \\t_2+2(u-h_i)-2(n-i)\end{matrix}\right\rangle,
\end{align}

\begin{align}\label{Fexp1}
F=F(\textbf{l},\textbf{h})=n(m+1)+\sum_{i=1}^{m}(2l_i-i)+\sum_{i=1}^{n}(2h_i-i),
\end{align}

\begin{align}\label{Eexp2}
E'=E'(x,u,d,\textbf{l},\textbf{h})&=f(t'_1+2d-m,n+1,m)\notag\\
&+\sum_{i=1}^{m}\left\langle\begin{matrix}
 t'_1+2d-m-(m-i+1)\\ t'_1+2(d-l_i)-2(m-i)\end{matrix}\right\rangle \notag\\
&+\sum_{i=1}^{n}\left\langle\begin{matrix}
t'_2+2u-n-(n-i+1)\\ t'_2+2(u-h_i)-2(n-i)\end{matrix}\right\rangle,
\end{align}

\begin{align}\label{Fexp2}
F'=F'(\textbf{l},\textbf{h})=(n+1)m+\sum_{i=1}^{m}(2l_i-i)+\sum_{i=1}^{n}(2h_i-i),
\end{align}
where $f(t,x,y)=\frac{xy(2t+x-y)}{2}$,  $t_1=2(x+u-n-e)+1$, $t_2=2(x+d-m-e)$, $t'_1=2(x+u-n-e')+1$, $t'_2=2(x+d-m+1-e')$, and $l_0=h_0=0$ by convention.

\begin{thm}[Formula for a region of type A]\label{typeAthm} For non-negative integers $x,d,m$ ($d\geq m$) and a sequence $(l_i)_{i=1}^{m}$ of positive integers between $1$ and $d$, we have
\begin{equation}\label{TypeAformula}
\M(A_{x,d}(l_1,l_2,\dots,l_m))=P(x,0,d,(l_i)_{i=1}^{m},\emptyset).
\end{equation}
\end{thm}

\begin{thm}[Formula for a region of type B]\label{typeBthm}
For non-negative integers $x,d,m$ ($d\geq m$) and a sequence $(l_i)_{i=1}^{m}$ of positive integers between $1$ and $d$, we have
\begin{equation}\label{TypeBformula}
\M(B_{x,d}(l_1,l_2,\dots,l_m))=Q(x,0,d,(l_i)_{i=1}^{m},\emptyset).
\end{equation}
\end{thm}

\begin{thm}[Formula for a region of type C]\label{typeCthm}For non-negative integers $x,u,n$ ($u\geq n$) and a sequence $(h_j)_{j=1}^{n}$ of positive integers between $1$ and $u$, we have
\begin{equation}\label{TypeCformula}
\M(C_{x,u}(h_1,h_2,\dots,h_n))=Q(x,u,0,\emptyset,(h_i)_{i=1}^{n}).
\end{equation}
\end{thm}

\begin{thm}[Formula for a region of type D]\label{typeDthm}For non-negative integers $x,u,n$ ($u\geq n$) and a sequence $(h_j)_{j=1}^{n}$ of positive integers between $1$ and $u$, we have
\begin{equation}\label{TypeDformula}
\M(D_{x,u}(h_1,h_2,\dots,h_n))=P(x,u,0,\emptyset,(h_i)_{i=1}^{n}).
\end{equation}
\end{thm}

\begin{figure}\centering
\setlength{\unitlength}{3947sp}%
\begingroup\makeatletter\ifx\SetFigFont\undefined%
\gdef\SetFigFont#1#2#3#4#5{%
  \reset@font\fontsize{#1}{#2pt}%
  \fontfamily{#3}\fontseries{#4}\fontshape{#5}%
  \selectfont}%
\fi\endgroup%
\resizebox{!}{11cm}{
\begin{picture}(0,0)%
\includegraphics{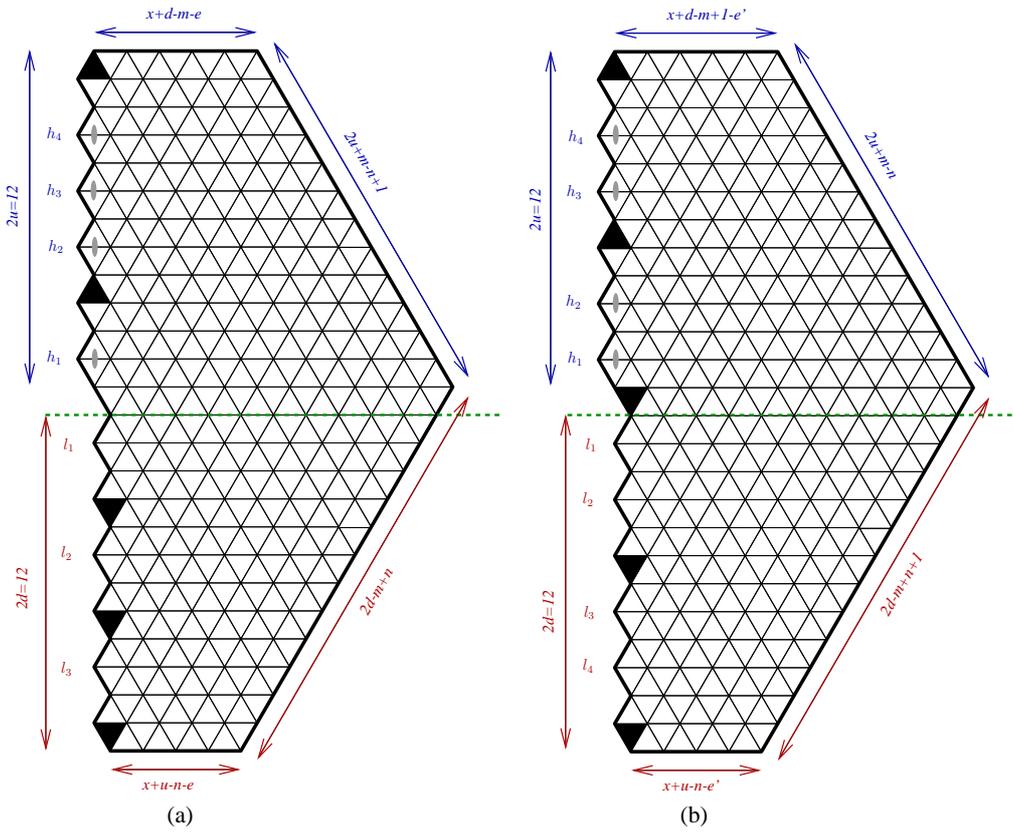}%
\end{picture}%

\begin{picture}(12879,10499)(1074,-11514)
\put(8416,-6684){\makebox(0,0)[lb]{\smash{{\SetFigFont{12}{14.4}{\rmdefault}{\mddefault}{\itdefault}{\color[rgb]{.69,0,0}$l_1$}%
}}}}
\put(8379,-7374){\makebox(0,0)[lb]{\smash{{\SetFigFont{12}{14.4}{\rmdefault}{\mddefault}{\itdefault}{\color[rgb]{.69,0,0}$l_2$}%
}}}}
\put(8394,-8806){\makebox(0,0)[lb]{\smash{{\SetFigFont{12}{14.4}{\rmdefault}{\mddefault}{\itdefault}{\color[rgb]{.69,0,0}$l_3$}%
}}}}
\put(8191,-5596){\makebox(0,0)[lb]{\smash{{\SetFigFont{12}{14.4}{\rmdefault}{\mddefault}{\itdefault}{\color[rgb]{0,0,.69}$h_1$}%
}}}}
\put(8169,-4876){\makebox(0,0)[lb]{\smash{{\SetFigFont{12}{14.4}{\rmdefault}{\mddefault}{\itdefault}{\color[rgb]{0,0,.69}$h_2$}%
}}}}
\put(8176,-3489){\makebox(0,0)[lb]{\smash{{\SetFigFont{12}{14.4}{\rmdefault}{\mddefault}{\itdefault}{\color[rgb]{0,0,.69}$h_3$}%
}}}}
\put(8199,-2791){\makebox(0,0)[lb]{\smash{{\SetFigFont{12}{14.4}{\rmdefault}{\mddefault}{\itdefault}{\color[rgb]{0,0,.69}$h_4$}%
}}}}
\put(8379,-9481){\makebox(0,0)[lb]{\smash{{\SetFigFont{12}{14.4}{\rmdefault}{\mddefault}{\itdefault}{\color[rgb]{.69,0,0}$l_4$}%
}}}}
\put(1861,-6684){\makebox(0,0)[lb]{\smash{{\SetFigFont{12}{14.4}{\rmdefault}{\mddefault}{\itdefault}{\color[rgb]{.69,0,0}$l_1$}%
}}}}
\put(1831,-8056){\makebox(0,0)[lb]{\smash{{\SetFigFont{12}{14.4}{\rmdefault}{\mddefault}{\itdefault}{\color[rgb]{.69,0,0}$l_2$}%
}}}}
\put(1831,-9519){\makebox(0,0)[lb]{\smash{{\SetFigFont{12}{14.4}{\rmdefault}{\mddefault}{\itdefault}{\color[rgb]{.69,0,0}$l_3$}%
}}}}
\put(1651,-5581){\makebox(0,0)[lb]{\smash{{\SetFigFont{12}{14.4}{\rmdefault}{\mddefault}{\itdefault}{\color[rgb]{0,0,.69}$h_1$}%
}}}}
\put(1674,-4171){\makebox(0,0)[lb]{\smash{{\SetFigFont{12}{14.4}{\rmdefault}{\mddefault}{\itdefault}{\color[rgb]{0,0,.69}$h_2$}%
}}}}
\put(1651,-3466){\makebox(0,0)[lb]{\smash{{\SetFigFont{12}{14.4}{\rmdefault}{\mddefault}{\itdefault}{\color[rgb]{0,0,.69}$h_3$}%
}}}}
\put(1651,-2746){\makebox(0,0)[lb]{\smash{{\SetFigFont{12}{14.4}{\rmdefault}{\mddefault}{\itdefault}{\color[rgb]{0,0,.69}$h_4$}%
}}}}
\end{picture}}
\caption{The shapes of the two hybrid halved hexagons.}\label{Fig:Twosidehole}
\end{figure}

\begin{figure}\centering
\includegraphics[width=13cm]{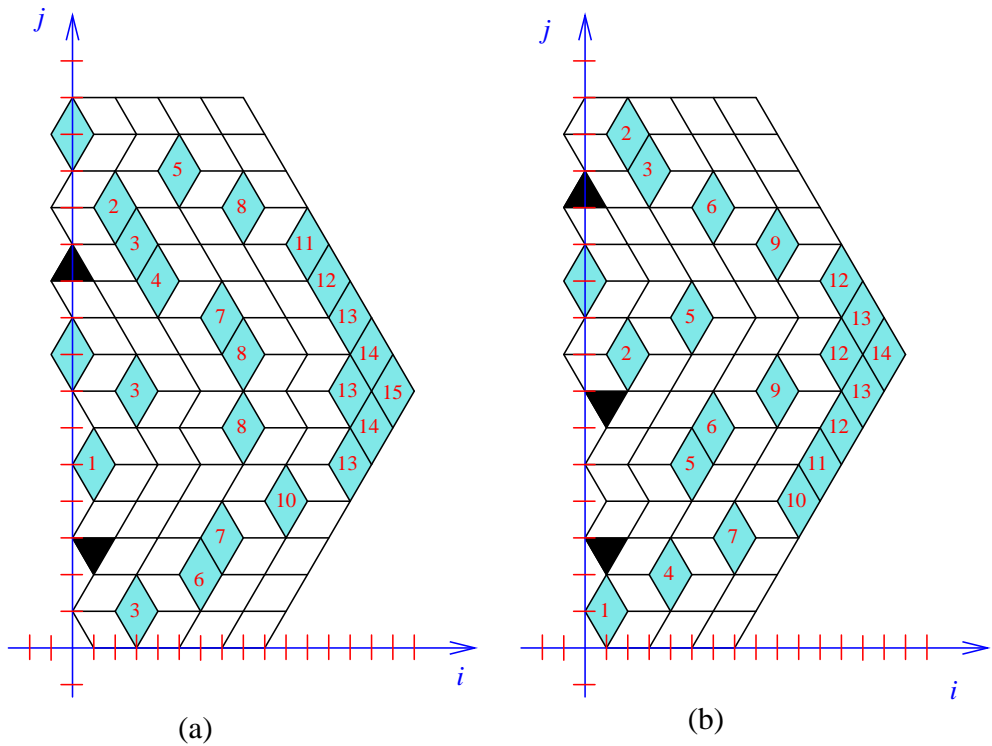}
\caption{Weight assignments for the two hybrid halved hexagons.}\label{Fig:TwosideholeTile}
\end{figure}

\medskip

We consider next a hybrid of  the $A$-type and $D$-type regions.  Our new region can roughly be viewed as a $D$-type region and (the vertical reflection of) an $A$-type region glued together along their bases. Let $x,u,d,m,$ and $n$ be non-negative integers, with $m\leq d$ and $n\leq u$. We consider the halved hexagons whose north, northeast, southeast and south sides have the side-lengths $x+d-m-e, 2u+m-n+1,2d-m+n,x+u-n+e$, respectively, where $e=\min (u-n, d-m)$. The west side is now the concatenation of two vertical zigzag paths with $2d$ and $2u$ steps as in Figure \ref{Fig:Twosidehole}(a). Strictly speaking there is an additional unit step between the two zigzag paths; this brings the length of the west side to $2u+2d+1$. We remove $u-n$ up-pointing unit triangles from the upper part of the west side and $d-m$ down-pointing unit triangles from the  lower part. Next, we assign the weights to the vertical lozenges of the region using the weight assignment $\wt_2$ as in Figure \ref{Fig:TwosideholeTile}(a). (The $j$-axis passes the upper-left vertex of the region, and the $i$-axis runs along the base.) Denote by $S_{x,u,d}((l_i)_{i=1}^{m}; (h_j)_{j=1}^{n})$ the resulting weighted region. 

We are also interested in a counterpart of the $S$-type region as follows. Our base halved hexagon has  side-lengths $x+d-m+1-e',2u+m-n, 2d-m+n+1,x+u-n-e', 2u+2d+1$, where $e'=\min (u-n, d-m+1)$. We still remove $u-n$ up-pointing and $d-m$ down-pointing unit triangles from the west side, as in the case of the $S$-type region. However, we remove an additional down-pointing triangle in the ``middle'' of the west side (in particular, we remove the only down-pointing unit triangle on the west side between the first upper and lower bumps). We use the weight assignment $\wt_2$ just as in the $S$-type regions. We denote by $T_{x,u,d}((l_i)_{i=1}^{m}; (h_j)_{j=1}^{n})$ the new weighted regions. See Figures \ref{Fig:Twosidehole}(b) and \ref{Fig:TwosideholeTile}(b) for examples. One can also viewed a $T$-type region as a hybrid version of a $C$-type region and (the vertical reflection of) a $B$-type region.

The TGFs of the $S$- and $T$-type regions are also given by simple product formulas. Moreover, their tiling generating function are generalizations of those of the $A$- and $D$-type regions and the $B$- and  $C$- regions, respectively.

\begin{thm}[Formula for region of type S]\label{typeSthm} Assume that $x,u,d,m,$ and $n$ are non-negative integers ($d\geq m$, $u\geq n$), and $(l_i)_{i=1}^{m}$ and $(h_j)_{j=1}^{n}$ are two sequences of positive integers between $1$ and $d$ and between $1$ and $u$, respectively. Then
\begin{equation}\label{TypeSformula}
\M(S_{x,u,d}((l_i)_{i=1}^{m}, (h_j)_{j=1}^{n})=P(x,u,d,(l_i)_{i=1}^{m}, (h_j)_{j=1}^{n}).
\end{equation}
\end{thm}

\begin{thm}[Formula for region of type T]\label{typeTthm}Assume that $x,u,d,m,$ and $n$ are non-negative integers ($d\geq m$, $u\geq n$), and $(l_i)_{i=1}^{m}$ and $(h_j)_{j=1}^{n}$ are two sequences of positive integers between $1$ and $d$ and between $1$ and $u$, respectively. Then
\begin{equation}\label{TypeTformula}
\M(T_{x,u,d}((l_i)_{i=1}^{m}, (h_j)_{j=1}^{n})=Q(x,u,d,(l_i)_{i=1}^{m}, (h_j)_{j=1}^{n}).
\end{equation}
\end{thm}

\begin{rmk}
We note that the unweighted versions of Theorems \ref{typeAthm}--\ref{typeTthm} were treated by Ciucu in \cite{Ciucu2}. Moreover, following the line of work in Ciucu's paper, one can obtain the tiling generating function of the symmetric hexagon with two families of holes on the symmetry axis, say by using our Theorems \ref{typeSthm} and \ref{typeTthm} along with Ciucu's Factorization Theorem \cite[Theorem 1.2]{Ciucu97}. We leave it as an exercise for the reader.
\end{rmk}

\section{Preliminaries}\label{Sec:Pre}

A \emph{perfect matching} of a graph is a collection of disjoint edges that covers all vertices of the graph. There is bijection between tilings of a region $R$ on the triangular lattice and perfect matchings of its \emph{(planar) dual graph} $G$ (i.e., the graph whose vertices are the unit triangles in $R$ and whose edges connect precisely two unit triangles sharing an edge). The weight of an edge in the  dual graph is equal to that of the corresponding lozenge in the region. We use the notation $\M(G)$ for the weighted sum of the perfect matchings of the graph $G$, where the weight of a perfect matching is the product of its edge-weights. We often call $\M(G)$ the \emph{matching generating function} of $G$.

We will use the following two versions of the Kuo condensation in our proofs.

\begin{lem}[Theorem 5.1 in \cite{Kuo}]\label{Kuolem1}
Let $G=(V_1,V_2,E)$ be a weighted plane bipartite graph in which $|V_1|=|V_2|$. Let vertices $u,v,w,s$ appear on a face of $G$, in that order. If $u,w \in V_1$ and $v,s\in V_2$, then
\begin{equation}\label{Kuoeq1}
\M(G)\M(G-\{u,v,w,s\})=\M(G-\{u,v\})\M(G-\{w,s\})+\M(G-\{u,s\})\M(G-\{v,w\}).
\end{equation}
\end{lem}

\begin{lem}[Theorem 5.3 in \cite{Kuo}]\label{Kuolem2}
Let $G=(V_1,V_2,E)$ be a weighted plane bipartite graph in which $|V_1|=|V_2|+1$. Let vertices $u,v,w,s$ appear on a face of $G$, in that order. If $u,v,w \in V_1$ and $s\in V_2$, then
\begin{equation}\label{Kuoeq2}
\M(G-\{v\})\M(G-\{u,w,s\})=\M(G-\{u\})\M(G-\{v,w,s\})+\M(G-\{w\})\M(G-\{u,v,s\}).
\end{equation}
\end{lem}


A \emph{forced lozenge} of a region $R$ is a lozenge that must be contained in any tiling of $R$. Assume that we remove $k$ forced lozenges $l_1,l_2,\dots,l_k$ from $R$ to obtain the region $R'$. Then we have
\begin{equation}
\M(R')=\left(\prod_{i=1}^{k}\wt(l_i)\right)^{-1} \cdot \M(R),
\end{equation}
where $\wt(l_i)$ is the weight of the lozenge $l_i$.

If the region $R$ admits a tiling, then $R$ must have the same number of up-pointing and down-pointing unit triangles.  We call such a region \emph{balanced}.  The following simple lemma is especially useful  in the enumeration of tilings as it allows us to decompose a large region into smaller ones.

\begin{lem}[Region-splitting Lemma \cite{Tri18,Tri19}]\label{RS}
Assume $R$ is a balanced region  and $Q$  is a subregion of $R$ satisfying two conditions:
\begin{enumerate}
\item The unit triangles in $Q$ that have an edge on the  boundary between $Q$ and $R\setminus Q$ have the same orientation (all are up-pointing or all are down-pointing);
\item Q is balanced.
\end{enumerate}
Then we have $\M(R)=\M(Q) \cdot \M(R\setminus Q)$.
\end{lem}

\section{Proofs Theorems \ref{quartthm1}--\ref{quartthm4}}\label{Sec:Quarter}

\begin{figure}\centering
\includegraphics[width=13cm]{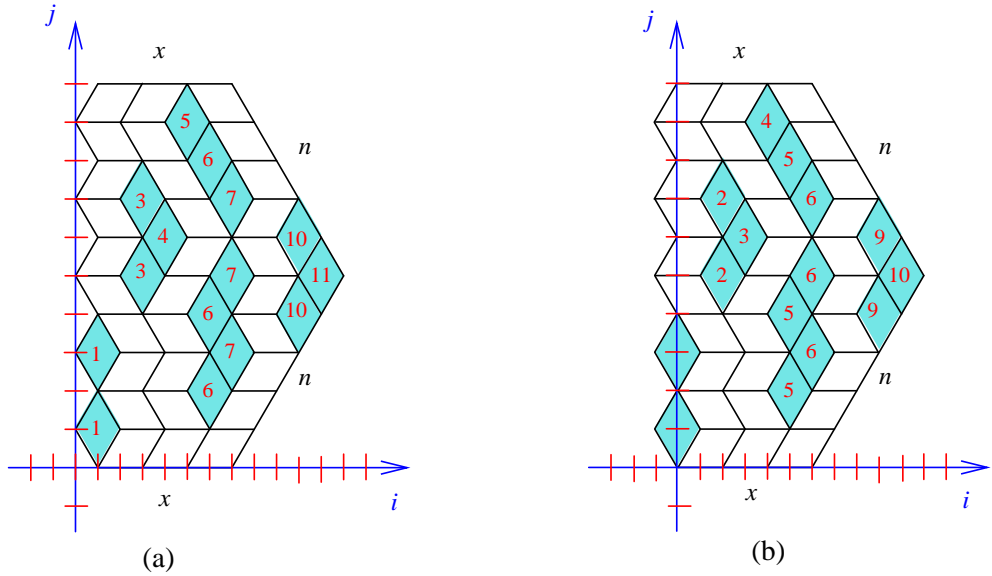}
\caption{Two ways to assign weights to the lozenges of a halved hexagon. The shaded lozenges intersected by the $j$-axis have weight $\frac{1}{2}$ while the ones with label $n$ have weight $\frac{q^{n}+q^{-n}}{2}$.}\label{Fig:halvedhex}
\end{figure}

Before we can prove Theorems \ref{quartthm1}--\ref{quartthm4} by induction, we prove a result on a weighted halved hexagon that serves as a base case. The halved hexagon is obtained by a symmetric hexagon of side-lengths $2x+1,n,n,2x+1,n,n$ (in counter-clockwise order, starting from the top side) by diving  along a vertical zigzag cut. We assign weights to lozenges of the halved hexagon as in Figure \ref{Fig:halvedhex}(a); this weight assignment is similar to that of $\wt_1$ from the quartered hexagons. Denote by $P_{n,x}$ the resulting weighted region.

The \emph{$q$-integer} $[n]_q$ is defined as $[n]_q=1+q+\cdots+q^{n-1}$. Then the \emph{$q$-factorial} is defined to be  the product of $q$ integers: $[n]_q!=[1]_q[2]_q\cdots[n]_q$.

\begin{lem}\label{proctorlem} Assume that $x$ and $n$ are non-negative integers. Then we have
\begin{align}\label{hheq1}
\M(P_{n,x})=\frac{2^{-n^2}q^{-\sum_{i=1}^{n}(2i-1)(2x+i)}}{[1]_{q^2}![3]_{q^2}!\cdots[2n-1]_{q^2}!}\prod_{i=1}^{n}[2(x+i)]_{q^2}\prod_{1\leq i <j\leq n}[2(2x+i+j)]_{q^2}[2(j-i)]_{q^2}.
\end{align}
\end{lem}

As in the case of the quartered hexagon, we consider a variant of $P_{n,x}$ by re-assigning the lozenge-weights as in Figure \ref{Fig:halvedhex}(b). In particular, the vertical lozenges with center at the point $(i,j)$ are still weighted by $\frac{q^{i}+q^{-i}}{2}$, except for the ones intersected by the $j$-axis, which have weight $1/2$ (this weighting is similar to that of $\wt_2$ from the quartered hexagons). Denote by $P'_{n,x}$ this new weighted region.

\begin{lem}\label{proctorlem2} Assume that $x$ and $n$ are non-negative integers. Then we have
\begin{align}\label{hheq2}
\M(P'_{n,x})=\frac{2^{-n^2}q^{-\sum_{i=1}^{n}(2i-1)(2x+i-1)}}{[1]_{q^2}![3]_{q^2}!\cdots[2n-1]_{q^2}!}\prod_{i=1}^{n}[2(x+i)-1]_{q^2}\prod_{1\leq i <j\leq n}[2(2x+i+j-1)]_{q^2}[2(j-i)]_{q^2}.
\end{align}
\end{lem}

Next, we prove Lemma \ref{proctorlem} (Lemma \ref{proctorlem2} can be proven in the same manner).

\begin{figure}\centering
\includegraphics[width=7cm]{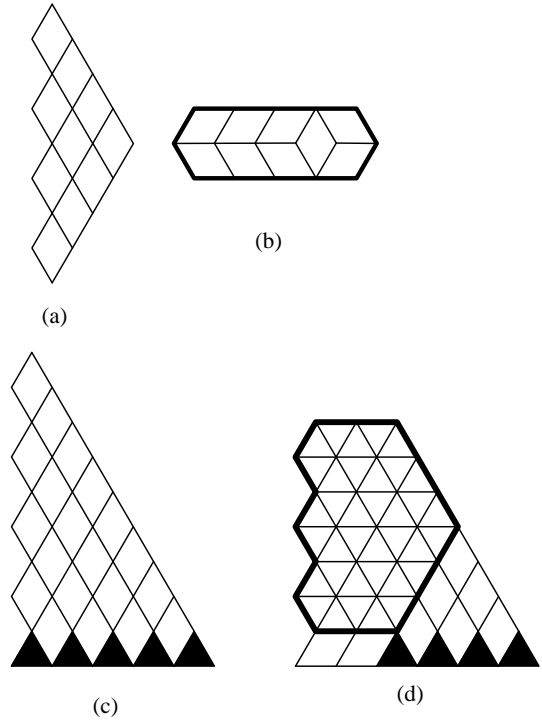}
\caption{Several special halved hexagons and quartered hexagons.}\label{Fig:WQspecial}
\end{figure}

\begin{figure}\centering
\includegraphics[width=10cm]{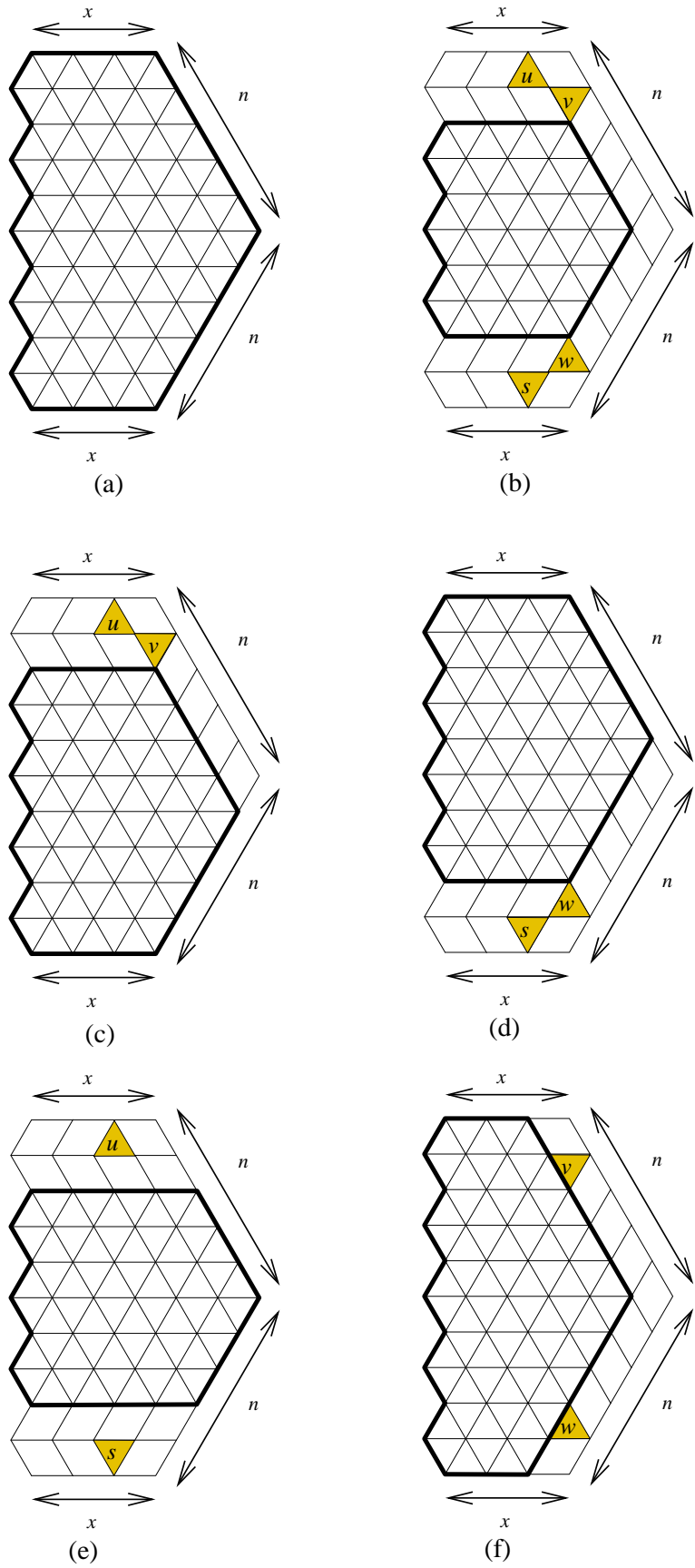}
\caption{Using Kuo condensation to obtain the recurrence for the tiling generating functions of halved hexagons.}\label{Fig:halved2}
\end{figure}

\begin{proof}[Proof of Lemma \ref{proctorlem}]
We prove this lemma by induction on $n+x$.  The base cases are the situations in which $n=0$, $n=1$, or $x=0$.

When $x=0$, the region has only one tiling (see Figure \ref{Fig:WQspecial}(a)). It is easy to verify our identity in this case.  When $n=0$, then the region is degenerate and our identity becomes ``1=1."   If $n=1$, then our region becomes a  hexagon of side-lengths $x,1,1,x,1,1$. It is easy to see that the hexagon has exactly $x$ tilings, each consists of one vertical  lozenge, $x$ left-tilted lozenges, and $x$ right-tilted lozenges (illustrated in Figure \ref{Fig:WQspecial}(b)). It is easy to calculate the TGF in this case, and then verify the identity.

For the induction step, we assume that $x>0, n>1$ and that the theorem holds for  any halved hexagon in which the sum of the $x$- and $n$-parameters is strictly than less than $x+n$. Apply Kuo condensation as in Lemma \ref{Kuolem1} to the dual graph $G$ of  $P_{n,x}$ with the four vertices $u,v,w,$ and $s$ chosen as in Figure \ref{Fig:halved2}(b). More precisely, the vertices $u,v,w,s$ of $G$ are indicated by the corresponding  unit triangles in the halved hexagon. The $u$- and $v$-triangles are the shaded ones in the upper-right corner of the region, and the $w$- and $s$-triangles are the shaded ones in the lower-right corner. Let us consider  the  region corresponding to the graph $G-\{u,v,w,s\}$.  The removal of the four shaded unit triangles produces some forced lozenges as in Figure \ref{Fig:halved2}(b). By removing these forced lozenges, we get a new and smaller halved hexagon, namely $P_{n-2,x}$ (see the region restricted by the bold contour). This implies that
\begin{equation}
\M(G-\{u,v,w,s\})=W_1\cdot\M(P_{n-2,x}),
\end{equation}
where $W_1$ is the product of the weights of the forced lozenges.
Considering the removal of forced lozenges as in Figure  \ref{Fig:halved2}(c)--(f), we get 
 \begin{equation}
\M(G-\{u,v\})=W_2\cdot\M(P_{n-1,x}),
\end{equation}
 \begin{equation}
\M(G-\{w,s\})=W_3\cdot\M(P_{n-1,x}),
\end{equation}
 \begin{equation}
\M(G-\{u,s\})=W_4\cdot\M(P_{n-2,x+1}),
\end{equation}
 \begin{equation}
\M(G-\{v,w\})=W_5\cdot\M(P_{n,x-1}), 
\end{equation}
where the $W_2, W_3, W_4,$ and $W_5$ are the products of the weights of the forced lozenges in their respective regions. Plugging  the above five equations into the recurrence in Lemma \ref{Kuolem1}, we get a recurrence for the TGFs of halved hexagons:
\begin{align}\label{hhrecurrence}
W_1 \cdot \M(P_{n,x})\M(P_{n-2,x})=W_2W_3&\M(P_{n-1,x})\M(P_{n-1,x})\notag\\
&+W_4W_5\cdot\M(P_{n-2,x+1})\M(P_{n,x-1}).
\end{align}

 We now carefully investigate the forced lozenges. We can see that only the vertical ones contribute to the factors $W_i$, for $i=1,\dots,5$ (as the others have weight $1$). First observe that $W_4=1$. Moreover, each weight of a vertical forced lozenge (except for the rightmost one) appears exactly one in each of $W_1$, $W_2\cdot W_3$, and $W_5$. The rightmost vertical forced lozenge, which we denote by $l_0$, appears exactly once in $W_1$ and $W_5$, but twice in the product $W_2\cdot W_3$. Therefore
\begin{equation}
\frac{W_2\cdot W_3}{W_1}=\wt(l_0)
\end{equation}
and
\begin{equation}
\frac{W_4\cdot W_5}{W_1}=\frac{W_5}{W_1}=1,
\end{equation}
where $\wt(l_0)$ is the weight of the lozenge $l_0$.
By definition, we have $\wt(l_0)=\frac{q^{2x+n}+q^{-(2x+n)}}{2}$. The recurrence (\ref{hhrecurrence}) now reduces to
\begin{align}
\M(P_{n,x})\M(P_{n-2,x})=\frac{q^{2x+n}+q^{-2x-n}}{2}\M(P_{n-1,x})\M(P_{n-1,x})+\M(P_{n-2,x+1})\M(P_{n,x-1}).
\end{align}
It is easy to see that the statistic $n+x$ of the last five regions in the above recurrence are all less than that of the first one. It is routine to check that the expression on the right-hand side of  (\ref{hheq1}) satisfies this recurrence, and our theorem follows from the induction principle.
\end{proof}

\begin{proof}[Proof of Theorem \ref{quartthm1}]
Assume that $l$ is the largest index such that there is not a dent at the position $s_{l}-1$  on the base of the quartered hexagon $R^{1}_{x}(s_1,\dots,s_n)$, where $s_{n+1}=n+x+1$ by convention. It means that $n-l+1$ is the size of the maximal cluster of dents attaching to the lower-right corner of the region. 
We now prove (\ref{quarterformula1}) by induction on the statistic $p:=x+n+l$. The base cases are the situations in which $n=0$, $x=0$, or $l=1$.

When $n=0$, our region is a degenerate one. By convention, the TGF is taken to be $1$, and our identity becomes ``1=1".
When $x=0$, the region only has one tiling consisting of all vertical lozenges (see Figure \ref{Fig:WQspecial}(c)). The weight of this tiling is exactly the TGF of the region. It is easy to verify that  (\ref{quarterformula1}) holds in this case.

When $l=1$, all dents are contiguous and form a large cluster of size $n$ attached to the lower-right corner of the region. By removing forced lozenges, we obtained the halved hexagon $P_{x,n-1}$ (see Figure \ref{Fig:WQspecial}(d)). Then  (\ref{quarterformula1})  follows from Lemma \ref{proctorlem}.

\medskip

For the induction step we assume that $x,n>0$ and $l>1$ and that (\ref{quarterformula1}) holds for any region with the $p$-statistic strictly less than $x+n+l$.

If $l=n+1$, then we can remove forced lozenges with weight $1$ along the right side of the region to obtain a ``smaller" region. Here, we say that the $R^{1}$-type region $A$ is ``smaller'' than a region $B$ of the same type if the $p$-statistic in $A$ is less than that in $B$. Then (\ref{quarterformula1})  follows from the induction hypothesis. 

In the rest of the proof, we assume that $l\leq n$. We now apply Kuo condensation from Lemma \ref{Kuolem1}  to the dual graph $G$ of the region $R$ obtained from $R^{1}_{x}(s_1,\dots,s_n)$ by filling the leftmost dent ($s_1$) with a unit triangle. The vertices $u,v,w,s$ of $G$ are indicated by the corresponding unit triangles of $R$ in Figure \ref{Fig:WQKuo1}(a). In particular, the $u$-triangle is the up-pointing unit triangle at position $\alpha=s_{l}-1$, the $v$-triangle is at position $s_1$, and the $w$- and $s$-triangles are at the upper-right corner of the region.

\begin{figure}\centering
\setlength{\unitlength}{3947sp}%
\begingroup\makeatletter\ifx\SetFigFont\undefined%
\gdef\SetFigFont#1#2#3#4#5{%
  \reset@font\fontsize{#1}{#2pt}%
  \fontfamily{#3}\fontseries{#4}\fontshape{#5}%
  \selectfont}%
\fi\endgroup%
\resizebox{!}{12cm}{
\begin{picture}(0,0)%
\includegraphics{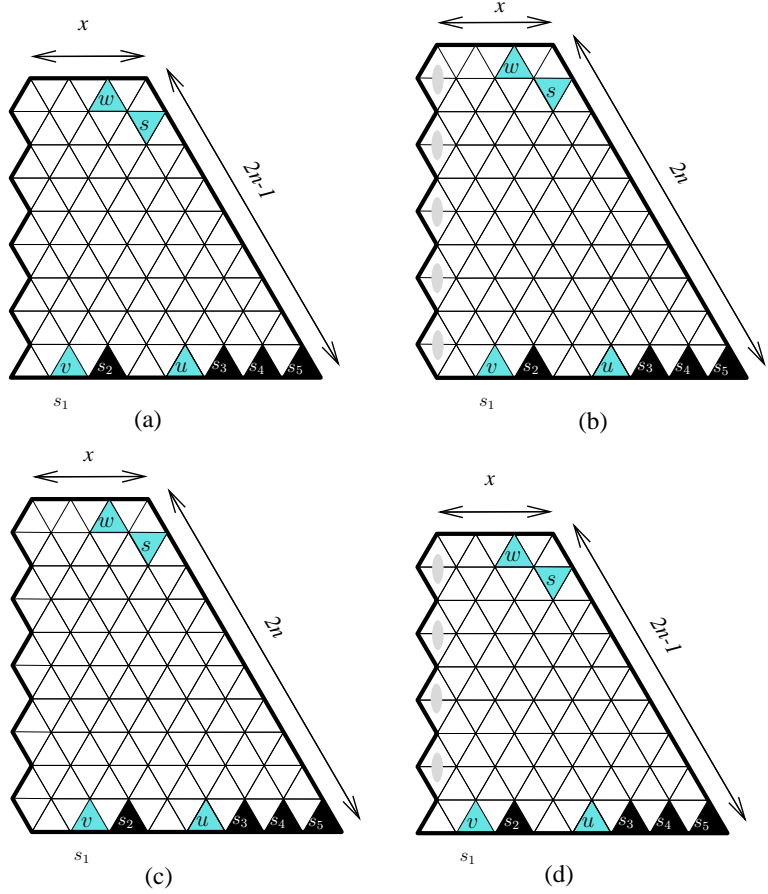}%
\end{picture}%
%
%

\begin{picture}(8044,9640)(2615,-8790)
\put(7681,-3249){\makebox(0,0)[lb]{\smash{{\SetFigFont{14}{16.8}{\rmdefault}{\mddefault}{\itdefault}{\color[rgb]{0,0,0}$v$}%
}}}}
\put(3582,-399){\makebox(0,0)[lb]{\smash{{\SetFigFont{14}{16.8}{\rmdefault}{\mddefault}{\itdefault}{\color[rgb]{0,0,0}$w$}%
}}}}
\put(4014,-659){\makebox(0,0)[lb]{\smash{{\SetFigFont{14}{16.8}{\rmdefault}{\mddefault}{\itdefault}{\color[rgb]{0,0,0}$s$}%
}}}}
\put(7861,-54){\makebox(0,0)[lb]{\smash{{\SetFigFont{14}{16.8}{\rmdefault}{\mddefault}{\itdefault}{\color[rgb]{0,0,0}$w$}%
}}}}
\put(8311,-306){\makebox(0,0)[lb]{\smash{{\SetFigFont{14}{16.8}{\rmdefault}{\mddefault}{\itdefault}{\color[rgb]{0,0,0}$s$}%
}}}}
\put(8311,-5512){\makebox(0,0)[lb]{\smash{{\SetFigFont{14}{16.8}{\rmdefault}{\mddefault}{\itdefault}{\color[rgb]{0,0,0}$s$}%
}}}}
\put(3402,-8087){\makebox(0,0)[lb]{\smash{{\SetFigFont{14}{16.8}{\rmdefault}{\mddefault}{\itdefault}{\color[rgb]{0,0,0}$v$}%
}}}}
\put(3582,-4892){\makebox(0,0)[lb]{\smash{{\SetFigFont{14}{16.8}{\rmdefault}{\mddefault}{\itdefault}{\color[rgb]{0,0,0}$w$}%
}}}}
\put(4032,-5144){\makebox(0,0)[lb]{\smash{{\SetFigFont{14}{16.8}{\rmdefault}{\mddefault}{\itdefault}{\color[rgb]{0,0,0}$s$}%
}}}}
\put(8693,-8092){\makebox(0,0)[lb]{\smash{{\SetFigFont{14}{16.8}{\rmdefault}{\mddefault}{\itdefault}{\color[rgb]{0,0,0}$u$}%
}}}}
\put(8896,-3241){\makebox(0,0)[lb]{\smash{{\SetFigFont{14}{16.8}{\rmdefault}{\mddefault}{\itdefault}{\color[rgb]{0,0,0}$u$}%
}}}}
\put(7478,-8100){\makebox(0,0)[lb]{\smash{{\SetFigFont{14}{16.8}{\rmdefault}{\mddefault}{\itdefault}{\color[rgb]{0,0,0}$v$}%
}}}}
\put(7879,-5252){\makebox(0,0)[lb]{\smash{{\SetFigFont{14}{16.8}{\rmdefault}{\mddefault}{\itdefault}{\color[rgb]{0,0,0}$w$}%
}}}}
\put(4396,-3239){\makebox(0,0)[lb]{\smash{{\SetFigFont{14}{16.8}{\rmdefault}{\mddefault}{\itdefault}{\color[rgb]{0,0,0}$u$}%
}}}}
\put(3181,-3247){\makebox(0,0)[lb]{\smash{{\SetFigFont{14}{16.8}{\rmdefault}{\mddefault}{\itdefault}{\color[rgb]{0,0,0}$v$}%
}}}}
\put(4617,-8079){\makebox(0,0)[lb]{\smash{{\SetFigFont{14}{16.8}{\rmdefault}{\mddefault}{\itdefault}{\color[rgb]{0,0,0}$u$}%
}}}}
\put(3106,-3599){\makebox(0,0)[lb]{\smash{{\SetFigFont{12}{14.4}{\rmdefault}{\mddefault}{\itdefault}{\color[rgb]{0,0,0}$s_1$}%
}}}}
\put(3791,-8067){\makebox(0,0)[lb]{\smash{{\SetFigFont{12}{14.4}{\rmdefault}{\mddefault}{\itdefault}{\color[rgb]{1,1,1}$s_2$}%
}}}}
\put(4999,-8052){\makebox(0,0)[lb]{\smash{{\SetFigFont{12}{14.4}{\rmdefault}{\mddefault}{\itdefault}{\color[rgb]{1,1,1}$s_3$}%
}}}}
\put(5396,-8067){\makebox(0,0)[lb]{\smash{{\SetFigFont{12}{14.4}{\rmdefault}{\mddefault}{\itdefault}{\color[rgb]{1,1,1}$s_4$}%
}}}}
\put(5801,-8075){\makebox(0,0)[lb]{\smash{{\SetFigFont{12}{14.4}{\rmdefault}{\mddefault}{\itdefault}{\color[rgb]{1,1,1}$s_5$}%
}}}}
\put(7606,-3601){\makebox(0,0)[lb]{\smash{{\SetFigFont{12}{14.4}{\rmdefault}{\mddefault}{\itdefault}{\color[rgb]{0,0,0}$s_1$}%
}}}}
\put(8070,-3229){\makebox(0,0)[lb]{\smash{{\SetFigFont{12}{14.4}{\rmdefault}{\mddefault}{\itdefault}{\color[rgb]{1,1,1}$s_2$}%
}}}}
\put(9278,-3214){\makebox(0,0)[lb]{\smash{{\SetFigFont{12}{14.4}{\rmdefault}{\mddefault}{\itdefault}{\color[rgb]{1,1,1}$s_3$}%
}}}}
\put(9675,-3229){\makebox(0,0)[lb]{\smash{{\SetFigFont{12}{14.4}{\rmdefault}{\mddefault}{\itdefault}{\color[rgb]{1,1,1}$s_4$}%
}}}}
\put(3570,-3227){\makebox(0,0)[lb]{\smash{{\SetFigFont{12}{14.4}{\rmdefault}{\mddefault}{\itdefault}{\color[rgb]{1,1,1}$s_2$}%
}}}}
\put(4778,-3212){\makebox(0,0)[lb]{\smash{{\SetFigFont{12}{14.4}{\rmdefault}{\mddefault}{\itdefault}{\color[rgb]{1,1,1}$s_3$}%
}}}}
\put(10080,-3237){\makebox(0,0)[lb]{\smash{{\SetFigFont{12}{14.4}{\rmdefault}{\mddefault}{\itdefault}{\color[rgb]{1,1,1}$s_5$}%
}}}}
\put(5175,-3227){\makebox(0,0)[lb]{\smash{{\SetFigFont{12}{14.4}{\rmdefault}{\mddefault}{\itdefault}{\color[rgb]{1,1,1}$s_4$}%
}}}}
\put(7403,-8452){\makebox(0,0)[lb]{\smash{{\SetFigFont{12}{14.4}{\rmdefault}{\mddefault}{\itdefault}{\color[rgb]{0,0,0}$s_1$}%
}}}}
\put(7867,-8080){\makebox(0,0)[lb]{\smash{{\SetFigFont{12}{14.4}{\rmdefault}{\mddefault}{\itdefault}{\color[rgb]{1,1,1}$s_2$}%
}}}}
\put(9075,-8065){\makebox(0,0)[lb]{\smash{{\SetFigFont{12}{14.4}{\rmdefault}{\mddefault}{\itdefault}{\color[rgb]{1,1,1}$s_3$}%
}}}}
\put(9472,-8080){\makebox(0,0)[lb]{\smash{{\SetFigFont{12}{14.4}{\rmdefault}{\mddefault}{\itdefault}{\color[rgb]{1,1,1}$s_4$}%
}}}}
\put(5580,-3235){\makebox(0,0)[lb]{\smash{{\SetFigFont{12}{14.4}{\rmdefault}{\mddefault}{\itdefault}{\color[rgb]{1,1,1}$s_5$}%
}}}}
\put(9877,-8088){\makebox(0,0)[lb]{\smash{{\SetFigFont{12}{14.4}{\rmdefault}{\mddefault}{\itdefault}{\color[rgb]{1,1,1}$s_5$}%
}}}}
\put(3327,-8439){\makebox(0,0)[lb]{\smash{{\SetFigFont{12}{14.4}{\rmdefault}{\mddefault}{\itdefault}{\color[rgb]{0,0,0}$s_1$}%
}}}}
\end{picture}}
\caption{An application of Kuo condensation to the quartered hexagons.}\label{Fig:WQKuo1}
\end{figure}

\begin{figure}\centering
\setlength{\unitlength}{3947sp}%
\begingroup\makeatletter\ifx\SetFigFont\undefined%
\gdef\SetFigFont#1#2#3#4#5{%
  \reset@font\fontsize{#1}{#2pt}%
  \fontfamily{#3}\fontseries{#4}\fontshape{#5}%
  \selectfont}%
\fi\endgroup%
\resizebox{!}{16cm}{
\begin{picture}(0,0)%
\includegraphics{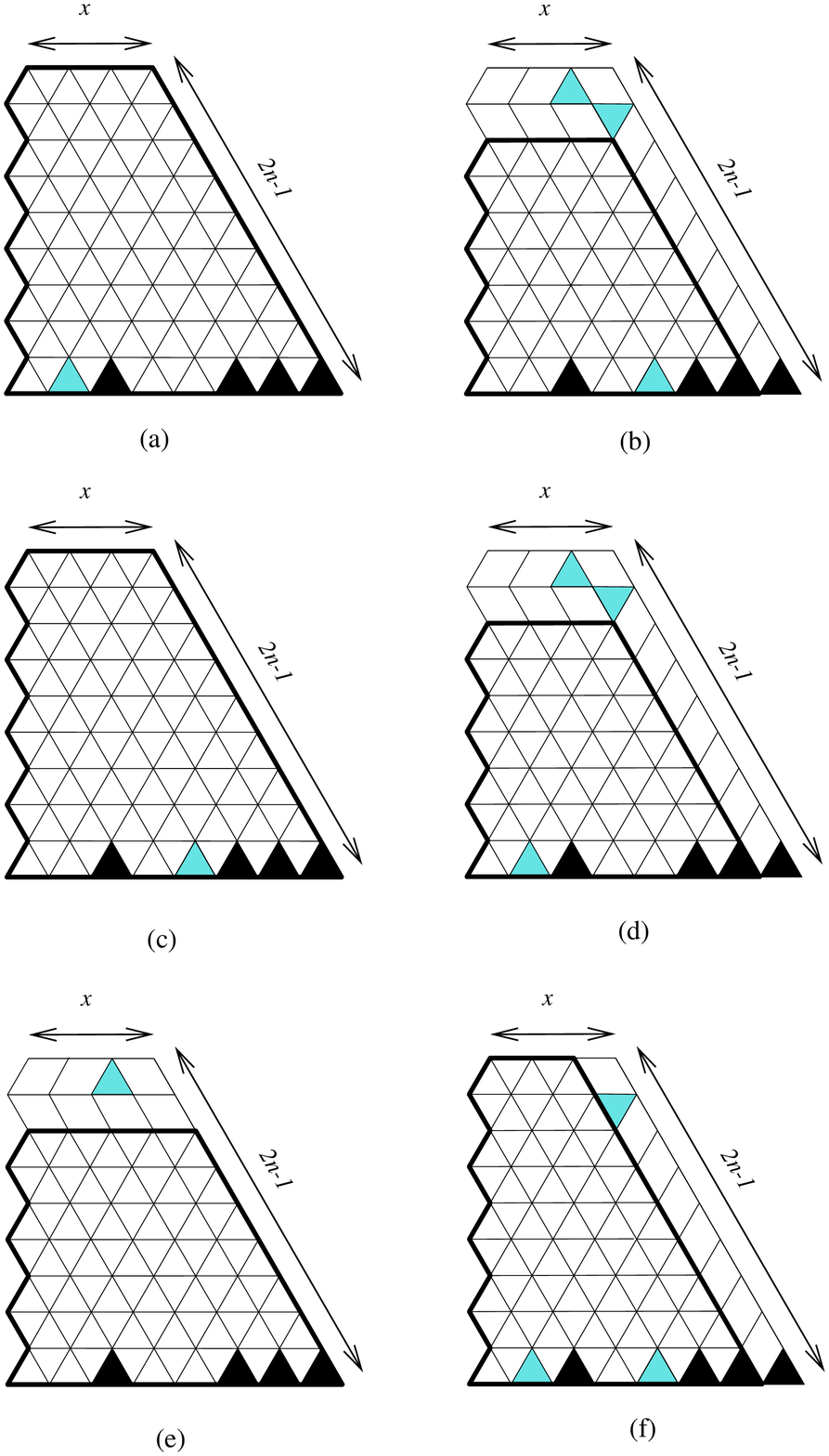}%
\end{picture}%
%
%

\begin{picture}(8072,14353)(2616,-13695)
\put(8514,-661){\makebox(0,0)[lb]{\smash{{\SetFigFont{14}{16.8}{\rmdefault}{\mddefault}{\itdefault}{\color[rgb]{0,0,0}$s$}%
}}}}
\put(8925,-12927){\makebox(0,0)[lb]{\smash{{\SetFigFont{14}{16.8}{\rmdefault}{\mddefault}{\itdefault}{\color[rgb]{0,0,0}$u$}%
}}}}
\put(8543,-10347){\makebox(0,0)[lb]{\smash{{\SetFigFont{14}{16.8}{\rmdefault}{\mddefault}{\itdefault}{\color[rgb]{0,0,0}$s$}%
}}}}
\put(7710,-12935){\makebox(0,0)[lb]{\smash{{\SetFigFont{14}{16.8}{\rmdefault}{\mddefault}{\itdefault}{\color[rgb]{0,0,0}$v$}%
}}}}
\put(3179,-3237){\makebox(0,0)[lb]{\smash{{\SetFigFont{14}{16.8}{\rmdefault}{\mddefault}{\itdefault}{\color[rgb]{0,0,0}$v$}%
}}}}

\put(3596,-10091){\makebox(0,0)[lb]{\smash{{\SetFigFont{14}{16.8}{\rmdefault}{\mddefault}{\itdefault}{\color[rgb]{0,0,0}$w$}%
}}}}
\put(4403,-7968){\makebox(0,0)[lb]{\smash{{\SetFigFont{14}{16.8}{\rmdefault}{\mddefault}{\itdefault}{\color[rgb]{0,0,0}$u$}%
}}}}
\put(8520,-5384){\makebox(0,0)[lb]{\smash{{\SetFigFont{14}{16.8}{\rmdefault}{\mddefault}{\itdefault}{\color[rgb]{0,0,0}$s$}%
}}}}
\put(7687,-7972){\makebox(0,0)[lb]{\smash{{\SetFigFont{14}{16.8}{\rmdefault}{\mddefault}{\itdefault}{\color[rgb]{0,0,0}$v$}%
}}}}
\put(8088,-5124){\makebox(0,0)[lb]{\smash{{\SetFigFont{14}{16.8}{\rmdefault}{\mddefault}{\itdefault}{\color[rgb]{0,0,0}$w$}%
}}}}
\put(8896,-3241){\makebox(0,0)[lb]{\smash{{\SetFigFont{14}{16.8}{\rmdefault}{\mddefault}{\itdefault}{\color[rgb]{0,0,0}$u$}%
}}}}
\put(8082,-401){\makebox(0,0)[lb]{\smash{{\SetFigFont{14}{16.8}{\rmdefault}{\mddefault}{\itdefault}{\color[rgb]{0,0,0}$w$}%
}}}}
\put(9704,-12915){\makebox(0,0)[lb]{\smash{{\SetFigFont{12}{14.4}{\rmdefault}{\mddefault}{\itdefault}{\color[rgb]{1,1,1}$s_4$}%
}}}}
\put(10109,-12923){\makebox(0,0)[lb]{\smash{{\SetFigFont{12}{14.4}{\rmdefault}{\mddefault}{\itdefault}{\color[rgb]{1,1,1}$s_5$}%
}}}}
\put(3571,-3226){\makebox(0,0)[lb]{\smash{{\SetFigFont{12}{14.4}{\rmdefault}{\mddefault}{\itdefault}{\color[rgb]{1,1,1}$s_2$}%
}}}}
\put(4779,-3211){\makebox(0,0)[lb]{\smash{{\SetFigFont{12}{14.4}{\rmdefault}{\mddefault}{\itdefault}{\color[rgb]{1,1,1}$s_3$}%
}}}}
\put(5176,-3226){\makebox(0,0)[lb]{\smash{{\SetFigFont{12}{14.4}{\rmdefault}{\mddefault}{\itdefault}{\color[rgb]{1,1,1}$s_4$}%
}}}}
\put(5581,-3234){\makebox(0,0)[lb]{\smash{{\SetFigFont{12}{14.4}{\rmdefault}{\mddefault}{\itdefault}{\color[rgb]{1,1,1}$s_5$}%
}}}}
\put(7606,-3601){\makebox(0,0)[lb]{\smash{{\SetFigFont{12}{14.4}{\rmdefault}{\mddefault}{\itdefault}{\color[rgb]{0,0,0}$s_1$}%
}}}}
\put(8070,-3229){\makebox(0,0)[lb]{\smash{{\SetFigFont{12}{14.4}{\rmdefault}{\mddefault}{\itdefault}{\color[rgb]{1,1,1}$s_2$}%
}}}}
\put(9278,-3214){\makebox(0,0)[lb]{\smash{{\SetFigFont{12}{14.4}{\rmdefault}{\mddefault}{\itdefault}{\color[rgb]{1,1,1}$s_3$}%
}}}}
\put(10080,-3237){\makebox(0,0)[lb]{\smash{{\SetFigFont{12}{14.4}{\rmdefault}{\mddefault}{\itdefault}{\color[rgb]{1,1,1}$s_5$}%
}}}}
\put(9675,-3229){\makebox(0,0)[lb]{\smash{{\SetFigFont{12}{14.4}{\rmdefault}{\mddefault}{\itdefault}{\color[rgb]{1,1,1}$s_4$}%
}}}}
\put(7612,-8324){\makebox(0,0)[lb]{\smash{{\SetFigFont{12}{14.4}{\rmdefault}{\mddefault}{\itdefault}{\color[rgb]{0,0,0}$s_1$}%
}}}}
\put(8076,-7952){\makebox(0,0)[lb]{\smash{{\SetFigFont{12}{14.4}{\rmdefault}{\mddefault}{\itdefault}{\color[rgb]{1,1,1}$s_2$}%
}}}}
\put(3113,-8328){\makebox(0,0)[lb]{\smash{{\SetFigFont{12}{14.4}{\rmdefault}{\mddefault}{\itdefault}{\color[rgb]{0,0,0}$s_1$}%
}}}}
\put(3577,-7956){\makebox(0,0)[lb]{\smash{{\SetFigFont{12}{14.4}{\rmdefault}{\mddefault}{\itdefault}{\color[rgb]{1,1,1}$s_2$}%
}}}}
\put(4785,-7941){\makebox(0,0)[lb]{\smash{{\SetFigFont{12}{14.4}{\rmdefault}{\mddefault}{\itdefault}{\color[rgb]{1,1,1}$s_3$}%
}}}}
\put(5182,-7956){\makebox(0,0)[lb]{\smash{{\SetFigFont{12}{14.4}{\rmdefault}{\mddefault}{\itdefault}{\color[rgb]{1,1,1}$s_4$}%
}}}}
\put(5587,-7964){\makebox(0,0)[lb]{\smash{{\SetFigFont{12}{14.4}{\rmdefault}{\mddefault}{\itdefault}{\color[rgb]{1,1,1}$s_5$}%
}}}}
\put(9284,-7937){\makebox(0,0)[lb]{\smash{{\SetFigFont{12}{14.4}{\rmdefault}{\mddefault}{\itdefault}{\color[rgb]{1,1,1}$s_3$}%
}}}}
\put(9681,-7952){\makebox(0,0)[lb]{\smash{{\SetFigFont{12}{14.4}{\rmdefault}{\mddefault}{\itdefault}{\color[rgb]{1,1,1}$s_4$}%
}}}}
\put(3114,-13284){\makebox(0,0)[lb]{\smash{{\SetFigFont{12}{14.4}{\rmdefault}{\mddefault}{\itdefault}{\color[rgb]{0,0,0}$s_1$}%
}}}}
\put(3584,-12919){\makebox(0,0)[lb]{\smash{{\SetFigFont{12}{14.4}{\rmdefault}{\mddefault}{\itdefault}{\color[rgb]{1,1,1}$s_2$}%
}}}}
\put(4792,-12904){\makebox(0,0)[lb]{\smash{{\SetFigFont{12}{14.4}{\rmdefault}{\mddefault}{\itdefault}{\color[rgb]{1,1,1}$s_3$}%
}}}}
\put(5189,-12919){\makebox(0,0)[lb]{\smash{{\SetFigFont{12}{14.4}{\rmdefault}{\mddefault}{\itdefault}{\color[rgb]{1,1,1}$s_4$}%
}}}}
\put(5594,-12927){\makebox(0,0)[lb]{\smash{{\SetFigFont{12}{14.4}{\rmdefault}{\mddefault}{\itdefault}{\color[rgb]{1,1,1}$s_5$}%
}}}}
\put(10086,-7960){\makebox(0,0)[lb]{\smash{{\SetFigFont{12}{14.4}{\rmdefault}{\mddefault}{\itdefault}{\color[rgb]{1,1,1}$s_5$}%
}}}}
\put(7635,-13287){\makebox(0,0)[lb]{\smash{{\SetFigFont{12}{14.4}{\rmdefault}{\mddefault}{\itdefault}{\color[rgb]{0,0,0}$s_1$}%
}}}}
\put(8099,-12915){\makebox(0,0)[lb]{\smash{{\SetFigFont{12}{14.4}{\rmdefault}{\mddefault}{\itdefault}{\color[rgb]{1,1,1}$s_2$}%
}}}}
\put(9307,-12900){\makebox(0,0)[lb]{\smash{{\SetFigFont{12}{14.4}{\rmdefault}{\mddefault}{\itdefault}{\color[rgb]{1,1,1}$s_3$}%
}}}}
\end{picture}}
\caption{Obtaining a recurrence for the TGF of a quartered hexagon.}\label{Fig:WQKuo2}
\end{figure}

The removal of the $u$-, $v$-, $w$- and $v$-triangles yields forced lozenges. By removing these forced lozenges, we get new $R^{1}$-type regions. The recurrence obtained from Kuo condensation tells us  that the product of the TGFs of the two regions in the top row of Figure \ref{Fig:WQKuo2} is equal to the product of the TGFs of the two regions in the middle row, plus the product of the TGFs of the two regions in the bottom row. Each region (after removing forced lozenges) is an $R^{1}$-type region.
We note that the weights of the forced lozenges all cancel out; therefore, we have the following recurrence for the TGFs of the quartered hexagon:
\begin{align}
\M(R^{1}_{x}(\{s_i\}_{i=1}^{n}))\M(R^{1}_{x}(\alpha\{s_i\}_{i=2}^{n-1}))=&\M(R^{1}_{x}(\alpha\{s_i\}_{i=2}^{n}))\M(R^{1}_{x}(\{s_i\}_{i=1}^{n-1}))\notag\\
&+\M(R^{1}_{x+1}(\{s_i\}_{i=2}^{n}))\M(R^{1}_{x-1}(\alpha\{s_i\}_{i=1}^{n-1})).
\end{align}
Here we use the notation $\alpha S$ for the ordered set obtained by including $\alpha$ and reordering the elements.

To finish the proof we verify that the expression on the right-hand side of (\ref{quarterformula1}) satisfies the same recurrence. This verification is  straightforward, but, due to the complexity of the formula, it is  not trivial. For completeness, we briefly show the verification. Denote by $f(\{s_1,s_2,\dots,s_n\})$ 
the expression on the right-hand side of (\ref{quarterformula1}). We would like to show that
\begin{align}\label{quarterformula1b}
\frac{f(\alpha\{s_i\}_{i=2}^{n})f(\{s_i\}_{i=1}^{n-1})}{f(\{s_i\}_{i=1}^{n})f(\alpha\{s_i\}_{i=2}^{n-1})}+\frac{f(\{s_i\}_{i=2}^{n})f(\alpha\{s_i\}_{i=1}^{n-1})}{f(\{s_i\}_{i=1}^{n})f(\alpha\{s_i\}_{i=2}^{n-1})}=1.
\end{align}
By definition, the powers of $2$ and $q$ in the  first term cancel out. Now we can write the first term as
\begin{align}
\frac{\Delta(\alpha\{s_i\}_{i=2}^{n})\cdot \Box(\alpha\{s_i\}_{i=2}^{n})\cdot \Delta(\{s_i\}_{i=1}^{n-1})\cdot \Box(\{s_i\}_{i=1}^{n-1})}{\Delta(\{s_i\}_{i=1}^{n})\cdot \Box(\{s_i\}_{i=1}^{n})\cdot \Delta(\alpha\{s_i\}_{i=2}^{n-1})\cdot\Box(\alpha\{s_i\}_{i=2}^{n-1})},
\end{align}
where $\Delta(S):=\prod_{1\leq i<j\leq n}[2(s_j-s_i)]_{q^2}$ and $\Box(S):=\prod_{1\leq i<j\leq n}[2(s_j+s_i-1)]_{q^2}$ for  an ordered set $S=\{s_1<s_2<\dots<s_n\}$.
By definition, we can simplify
\begin{align}
\frac{\Delta(\alpha\{s_i\}_{i=2}^{n})\Delta(\{s_i\}_{i=1}^{n-1})}{\Delta(\{s_i\}_{i=1}^{n})\Delta(\alpha\{s_i\}_{i=2}^{n-1})}=\frac{[2(s_n-\alpha)]_{q^2}}{[2(s_n-s_1)]_{q^2}}
\end{align}
and
\begin{align}
\frac{\Box(\alpha\{s_i\}_{i=2}^{n})\Box(\{s_i\}_{i=1}^{n-1})}{\Box(\{s_i\}_{i=1}^{n})\Box(\alpha\{s_i\}_{i=2}^{n-1})}=\frac{[2(s_n+\alpha-1)]_{q^2}}{[2(s_n+s_1-1)]_{q^2}}.
\end{align}
The first term now reduces to
\[\frac{[2(s_n-\alpha)]_{q^2}}{[2(s_n-s_1)]_{q^2}} \frac{[2(s_n+\alpha-1)]_{q^2}}{[2(s_n+s_1-1)]_{q^2}}.\]
Similarly, the second term is equal to
\[q^{4(s_n-\alpha)}\frac{[2(\alpha-s_1)]_{q^2}}{[2(s_n-s_1)]_{q^2}} \frac{[2(\alpha+s_1)]_{q^2}}{[2(s_n+s_1-1)]_{q^2}}.\]
The identity (\ref{quarterformula1}) is now equivalent to
\begin{align}
\frac{[2(s_n-\alpha)]_{q^2}}{[2(s_n-s_1)]_{q^2}} \frac{[2(s_n+\alpha-1)]_{q^2}}{[2(s_n+s_1-1)]_{q^2}}+q^{4(s_n-\alpha)}\frac{[2(\alpha-s_1)]_{q^2}}{[2(s_n-s_1)]_{q^2}} \frac{[2(\alpha+s_1)]_{q^2}}{[2(s_n+s_1-1)]_{q^2}}=1,
\end{align}
which is a true identity. This finishes the proof.
\end{proof}
The proofs of Theorem \ref{quartthm2}--\ref{quartthm4} are essentially the same as that of Theorem \ref{quartthm1}. We apply Kuo condensation as in Figures \ref{Fig:WQKuo1}(b)--(d) and appeal to Lemmas \ref{proctorlem} and \ref{proctorlem2} to verify the base cases. We omit these proofs here.

\section{Proofs Theorems \ref{typeAthm} -- \ref{typeTthm}}\label{Sec:Twosidehole}

Let $R$ be a region of type $A,$ $B,$ $C,$ or $D$. Denote by $h=h(R)$ the \emph{height} of $R$, i.e. the sum of side-lengths of the northeast and southeast sides of the region. We have $h=2u$ if $R$ is $A_{x,u}$ or $C_{x,u}$ and $h=2u+1$ if $R$ is $B_{x,u}$ or $D_{x,u}$.

\begin{proof}[Combined proof of Theorems \ref{typeAthm} and \ref{typeBthm}]

We re-write (\ref{TypeAformula}) and (\ref{TypeBformula}) as
\begin{align}\label{Aformula}
\M(A_{x,d}((l_i)_{i=1}^m)&=2^{-E}q^{-F}\frac{\prod_{1\leq i < j \leq m}[2(l_j-l_i)]_{q^2}}{\prod_{i=1}^{m}[2l_i-1]_{q^2}!}\notag\\
&\times \prod_{i=1}^{\lceil m/2\rceil}\prod_{j=1}^{2m-4i+3}[2(x+d-m)+2i+j-1]_{q^2} \notag\\
&\times \prod_{i=1}^{m}\prod_{j=1}^{l_i-i}[2((x+d-m)+i+j)]_{q^2}[2((x+d-m)-i-j+m+1)]_{q^2}
\end{align}
and
\begin{align}\label{Bformula}
\M(B_{x,d}((l_i)_{i=1}^{m})&=2^{-E'}q^{-F'}\frac{\prod_{1\leq i < j \leq m}[2(l_j-l_i)]_{q^2}}{\prod_{i=1}^{m}[2l_i]_{q^2}!}\notag\\
&\times \prod_{i=1}^{m}[2((x+d-m)+n+i)]_{q^2} [2((x+d-m)+i+1)]_{q^2}\notag\\
&\times\prod_{i=1}^{m}\prod_{j=1}^{i-1}[2(x+d-m)+i+j+1]_{q^2}\notag\\
&\times \prod_{i=1}^{m}\prod_{j=1}^{l_i-i}[2((x+d-m)+i+j+1)]_{q^2}[2((x+d-m)-i-j+m+1)]_{q^2},
\end{align}
where
\begin{align}\label{Eexp1b}
E=E(x,0,d,\textbf{l},\emptyset)&=\sum_{i=1}^{m}\left\langle\begin{matrix}
2x+1+2d-m-(m-i+1)\\ 2x+1+2(l_m-l_i)-2(m-i) \end{matrix}\right\rangle
\end{align}
\begin{align}\label{Fexp1b}
F=F(\textbf{l},\emptyset)=\sum_{i=1}^{m}(2l_i-i)
\end{align}
\begin{align}\label{Eexp2b}
E'=E'(x,0,d,\textbf{l},\emptyset)&=m(4x+4d-3m+3)\notag\\
&+\sum_{i=1}^{m}\left\langle\begin{matrix}
2x+1+2d-m-(m-i+1)\\ 2x+1+2(l_m-l_i)-2(m-i)\end{matrix}\right\rangle
\end{align}
\begin{align}\label{Fexp2b}
F'=F'(\textbf{l},\emptyset)=m+\sum_{i=1}^{m}(2l_i-i),
\end{align}
 and $l_0=0$ by convention.

We prove (\ref{Aformula}) and (\ref{Bformula}) at the same time by induction on $h+x$, where $h$ is the height of the region as define above.

The base cases are the situations in which $h=0$, $h=1$, or $m=0$.

When $m=0$ our region has a unique tiling, as shown in Figure \ref{Fig:OnesideBase}(a) for an $A$-type region and Figure \ref{Fig:OnesideBase}(a) for a $B$-type region. The case $h=0$ and $h=1$ reduce to the case $m=0$.

\begin{figure}\centering
\includegraphics[width=13cm]{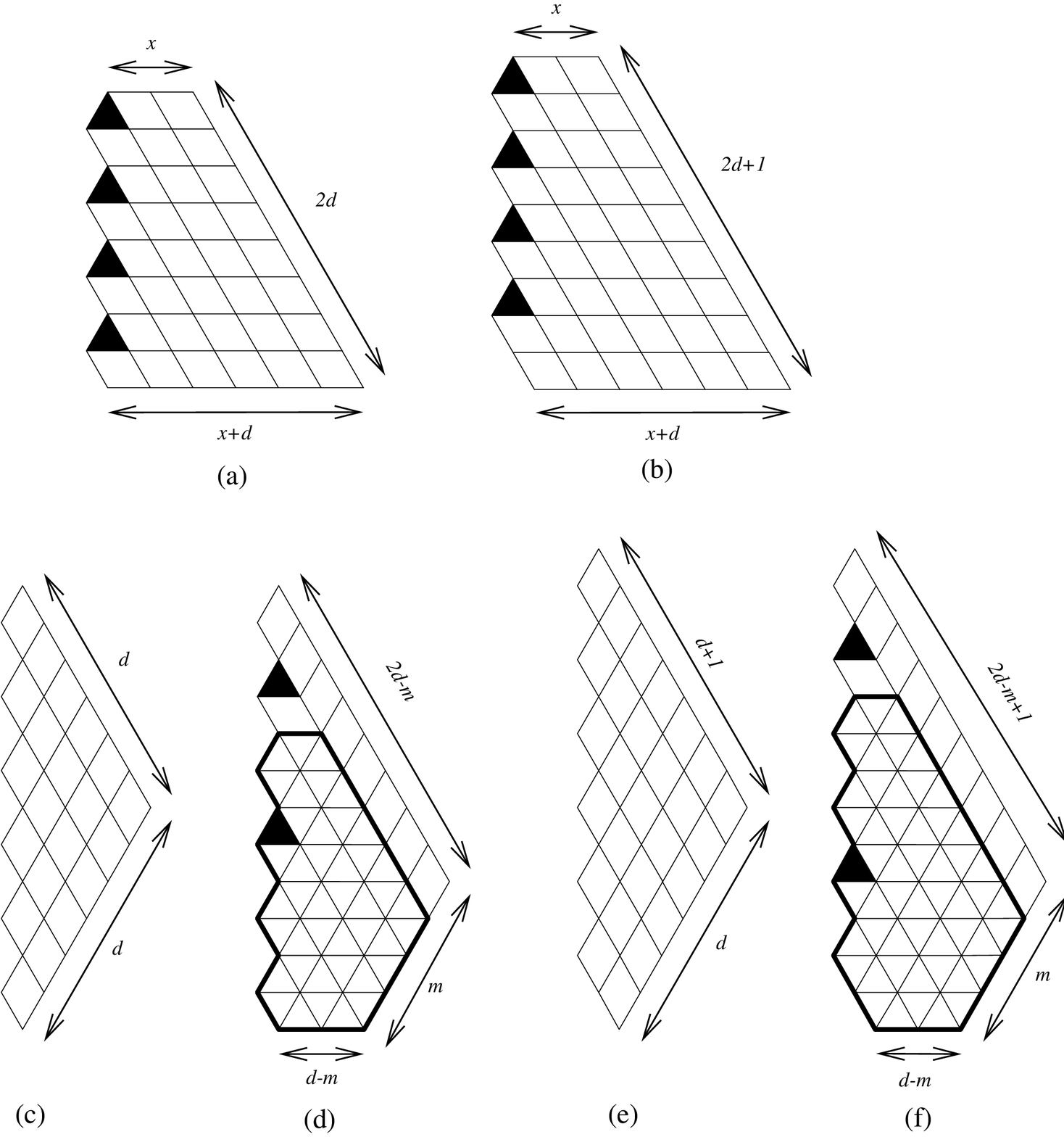}
\caption{Sepcial cases of proof of Theorems \ref{typeAthm} and \ref{typeBthm}.}\label{Fig:OnesideBase}
\end{figure}

\medskip

For the induction step, we assume that $h>1,m>0$ and that  (\ref{Aformula}) and (\ref{Bformula}) hold for all regions of type A or B in which the sum of the $x$- and $h$-parameters is strictly less than $x+h$.

We can assume that $x>0$. Indeed if $x=0$ and $d=m$, then our region has only one tiling consisting of vertical lozenges (see Figure \ref{Fig:OnesideBase}(c) for an $A$-type region, and Figure \ref{Fig:OnesideBase}(e) for a $B$-type region). We can then calculate the weight of this tiling and verify (\ref{Aformula}) and (\ref{Bformula}). If $x=0$ and $d>m$, then our region has some forced lozenges. By removing these lozenges, we get a `smaller'  region of the same type (indicated by the region restricted by the bold contour in Figures \ref{Fig:OnesideBase}(d) and (f), respectively for type $A$ and type $B$ regions).  In the remainder of the proof, we say that a region of type $A$ (or $B$) is ``smaller'' than another region of the same type if the sum of its $x$- and $h$-parameters is less than that of the other region. Therefore (\ref{Aformula}) and (\ref{Bformula}) follow from the induction hypothesis.

We can assume that $d=l_m$ (otherwise, we can remove forced lozenges on the top as in Figure \ref{Fig:Onesidehole2} to get a smaller region; then (\ref{Aformula}) and (\ref{Bformula}) follows from the induction hypothesis).
\begin{figure}\centering
\setlength{\unitlength}{3947sp}%
\begingroup\makeatletter\ifx\SetFigFont\undefined%
\gdef\SetFigFont#1#2#3#4#5{%
  \reset@font\fontsize{#1}{#2pt}%
  \fontfamily{#3}\fontseries{#4}\fontshape{#5}%
  \selectfont}%
\fi\endgroup%
\resizebox{!}{8cm}{
\begin{picture}(0,0)%
\includegraphics{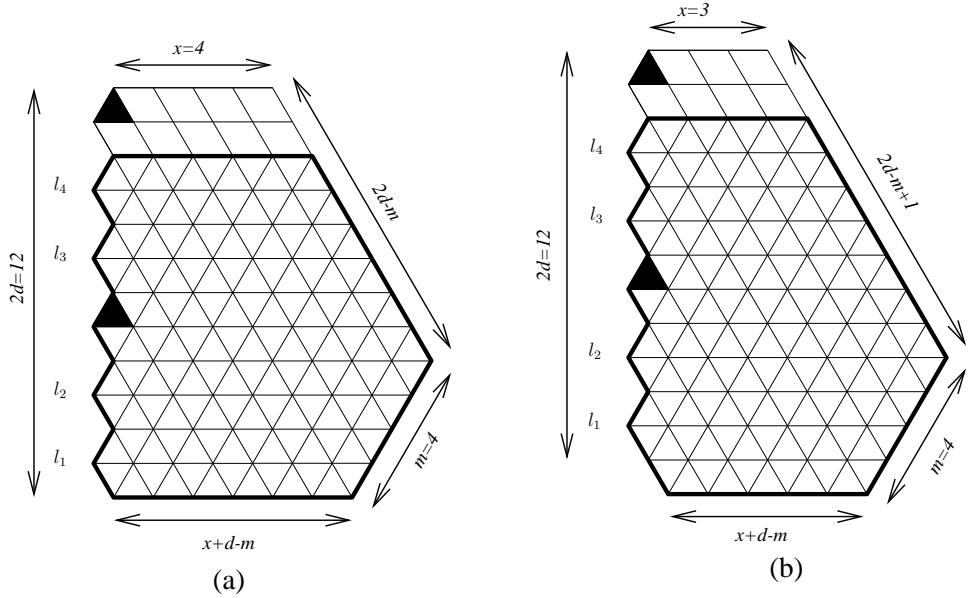}%
\end{picture}%

\begin{picture}(9943,6261)(903,-5404)
\put(1433,-4004){\makebox(0,0)[lb]{\smash{{\SetFigFont{12}{14.4}{\rmdefault}{\mddefault}{\itdefault}{\color[rgb]{0,0,0}$l_1$}%
}}}}
\put(1433,-3295){\makebox(0,0)[lb]{\smash{{\SetFigFont{12}{14.4}{\rmdefault}{\mddefault}{\itdefault}{\color[rgb]{0,0,0}$l_2$}%
}}}}
\put(1433,-1878){\makebox(0,0)[lb]{\smash{{\SetFigFont{12}{14.4}{\rmdefault}{\mddefault}{\itdefault}{\color[rgb]{0,0,0}$l_3$}%
}}}}
\put(1434,-1164){\makebox(0,0)[lb]{\smash{{\SetFigFont{12}{14.4}{\rmdefault}{\mddefault}{\itdefault}{\color[rgb]{0,0,0}$l_4$}%
}}}}
\put(6953,-768){\makebox(0,0)[lb]{\smash{{\SetFigFont{12}{14.4}{\rmdefault}{\mddefault}{\itdefault}{\color[rgb]{0,0,0}$l_4$}%
}}}}
\put(6953,-1480){\makebox(0,0)[lb]{\smash{{\SetFigFont{12}{14.4}{\rmdefault}{\mddefault}{\itdefault}{\color[rgb]{0,0,0}$l_3$}%
}}}}
\put(6938,-2890){\makebox(0,0)[lb]{\smash{{\SetFigFont{12}{14.4}{\rmdefault}{\mddefault}{\itdefault}{\color[rgb]{0,0,0}$l_2$}%
}}}}
\put(6931,-3612){\makebox(0,0)[lb]{\smash{{\SetFigFont{12}{14.4}{\rmdefault}{\mddefault}{\itdefault}{\color[rgb]{0,0,0}$l_1$}%
}}}}
\end{picture}}
\caption{Obtaining smaller regions of type A and type B by removing forced lozenges.} \label{Fig:Onesidehole2}
\end{figure}

\medskip

For the rest of the proof we may now assume that $m,x>0$ and $d=l_m$.

\medskip

We next consider the region $A=A_{x,d}((l_i)_{i=1}^{m})$ with $l_1=1$. 

If $d=m=1$, then the region is simply the hexagon of side-lengths $x,1,1,x,1,1$. This hexagon has exactly $x$ tilings; each tiling consists of one vertical lozenge, $x$ left-tilted lozenges, and $x$ right-tilted lozenges. It is easy to calculate the tiling generating function of the hexagon and verify (\ref{Aformula}) in this case. We now can assume that $d>1$, and as $l_m=d$ and $l_1=1$, we have $m>1$. 

We now apply Kuo condensation from Lemma \ref{Kuolem1} to the dual graph $G$ of the region $A=A_{x,d}((l_i)_{i=1}^{m})$. The four vertices $u,v,w,$ and $s$ are represented by their corresponding triangles in Figure \ref{Fig:Onesidehole3}(b). In particular, the $u$- and $v$-triangles are the shaded unit triangles in the upper-right corner of the region, and the $w$- and $s$-triangles are the shaded ones in the lower-right corner. 


We obtain a recurrence from Kuo condensation (the regions are shown in Figure \ref{Fig:Onesidehole3}):
\begin{align}\label{Arecurrence1}
\M(A_{x,d}((l_i)_{i=1}^{m}))\M(A_{x,d-2}((l_i-1)_{i=2}^{m-1}))&=\frac{W_2\cdot W_3}{W_1}\cdot\M(A_{x,d-1}((l_i)_{i=1}^{m-1}))\M(A_{x,d-1}((l_i-1)_{i=2}^{m}))\notag\\
&+\frac{W_4\cdot W_5}{W1}\cdot\M(A_{x+1,d-2}((l_i-1)_{i=2}^{m-1}))\M(A_{x-1,d}((l_i)_{i=1}^{m})),
\end{align}
where $W_1,\dots,W_5$ are the products of weights of forced lozenges in figures (b)--(f), respectively.

As in the proof of Lemma \ref{proctorlem}, we have
\begin{equation}
\frac{W_2\cdot W_3}{W_1}=\wt(l_0)
\end{equation}
and
\begin{equation}
\frac{W_4\cdot W_5}{W_1}=\frac{W_5}{W_1}=1,
\end{equation}
where $\wt(l_0)$ is the weight of the vertical lozenge $l_0$ at the right-most corner of the region.
By definition, we have $\wt(l_0)=\frac{q^{2x+2d-m}+q^{-(2x+2d-m)}}{2}$. Therefore, our recurrence above can be rewritten as:
\begin{align}\label{Arecurrence2}
\M(A_{x,d}((l_i)_{i=1}^{m}))&\M(A_{x,d-2}((l_i-1)_{i=2}^{m-1}))&\notag\\
=&\frac{q^{2x+2d-m}+q^{-(2x+2d-m)}}{2}\cdot\M(A_{x,d-1}((l_i)_{i=1}^{m-1}))\M(A_{x,d-1}((l_i-1)_{i=2}^{m}))\notag\\
&+\M(A_{x+1,d-2}((l_i-1)_{i=2}^{m-1}))\M(A_{x-1,d}((l_i)_{i=1}^{m})).
\end{align}
One can verify that the sum of the $h$- and $x$-parameters in the last five regions are all strictly less than that in the first one.

\begin{figure}\centering
\setlength{\unitlength}{3947sp}%
\begingroup\makeatletter\ifx\SetFigFont\undefined%
\gdef\SetFigFont#1#2#3#4#5{%
  \reset@font\fontsize{#1}{#2pt}%
  \fontfamily{#3}\fontseries{#4}\fontshape{#5}%
  \selectfont}%
\fi\endgroup%
\resizebox{!}{22cm}{
\begin{picture}(0,0)%
\includegraphics{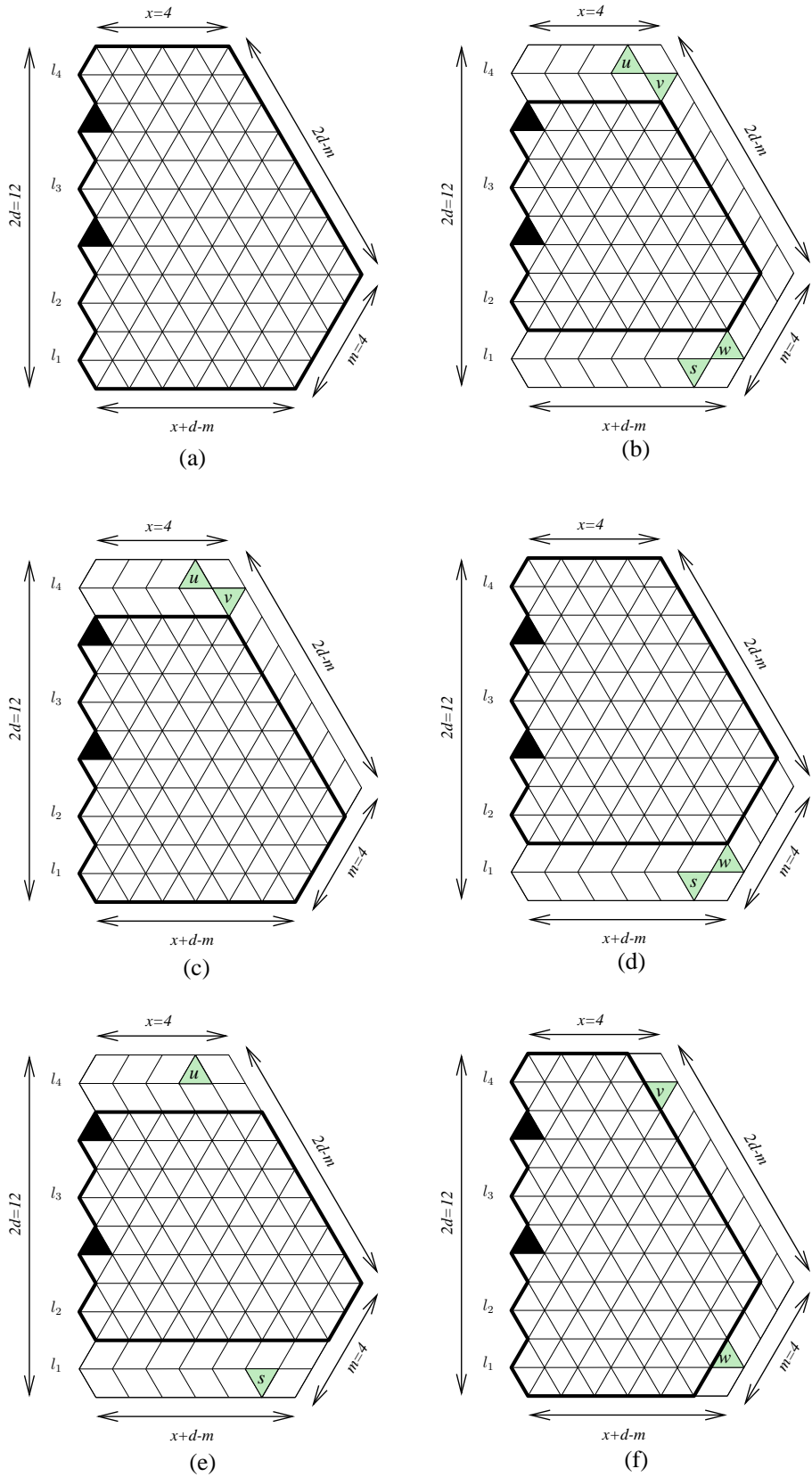}%
\end{picture}
\begin{picture}(9964,18357)(902,-17877)
\put(1433,-4004){\makebox(0,0)[lb]{\smash{{\SetFigFont{12}{14.4}{\rmdefault}{\mddefault}{\itdefault}{\color[rgb]{0,0,0}$l_1$}%
}}}}
\put(1433,-3295){\makebox(0,0)[lb]{\smash{{\SetFigFont{12}{14.4}{\rmdefault}{\mddefault}{\itdefault}{\color[rgb]{0,0,0}$l_2$}%
}}}}
\put(1433,-1878){\makebox(0,0)[lb]{\smash{{\SetFigFont{12}{14.4}{\rmdefault}{\mddefault}{\itdefault}{\color[rgb]{0,0,0}$l_3$}%
}}}}
\put(1433,-460){\makebox(0,0)[lb]{\smash{{\SetFigFont{12}{14.4}{\rmdefault}{\mddefault}{\itdefault}{\color[rgb]{0,0,0}$l_4$}%
}}}}
\put(6753,-3986){\makebox(0,0)[lb]{\smash{{\SetFigFont{12}{14.4}{\rmdefault}{\mddefault}{\itdefault}{\color[rgb]{0,0,0}$l_1$}%
}}}}
\put(6753,-3277){\makebox(0,0)[lb]{\smash{{\SetFigFont{12}{14.4}{\rmdefault}{\mddefault}{\itdefault}{\color[rgb]{0,0,0}$l_2$}%
}}}}
\put(6753,-1860){\makebox(0,0)[lb]{\smash{{\SetFigFont{12}{14.4}{\rmdefault}{\mddefault}{\itdefault}{\color[rgb]{0,0,0}$l_3$}%
}}}}
\put(6753,-442){\makebox(0,0)[lb]{\smash{{\SetFigFont{12}{14.4}{\rmdefault}{\mddefault}{\itdefault}{\color[rgb]{0,0,0}$l_4$}%
}}}}
\put(6752,-16507){\makebox(0,0)[lb]{\smash{{\SetFigFont{12}{14.4}{\rmdefault}{\mddefault}{\itdefault}{\color[rgb]{0,0,0}$l_1$}%
}}}}
\put(6752,-15798){\makebox(0,0)[lb]{\smash{{\SetFigFont{12}{14.4}{\rmdefault}{\mddefault}{\itdefault}{\color[rgb]{0,0,0}$l_2$}%
}}}}
\put(6752,-14381){\makebox(0,0)[lb]{\smash{{\SetFigFont{12}{14.4}{\rmdefault}{\mddefault}{\itdefault}{\color[rgb]{0,0,0}$l_3$}%
}}}}
\put(6752,-12963){\makebox(0,0)[lb]{\smash{{\SetFigFont{12}{14.4}{\rmdefault}{\mddefault}{\itdefault}{\color[rgb]{0,0,0}$l_4$}%
}}}}
\put(1432,-10375){\makebox(0,0)[lb]{\smash{{\SetFigFont{12}{14.4}{\rmdefault}{\mddefault}{\itdefault}{\color[rgb]{0,0,0}$l_1$}%
}}}}
\put(1432,-9666){\makebox(0,0)[lb]{\smash{{\SetFigFont{12}{14.4}{\rmdefault}{\mddefault}{\itdefault}{\color[rgb]{0,0,0}$l_2$}%
}}}}
\put(1432,-8249){\makebox(0,0)[lb]{\smash{{\SetFigFont{12}{14.4}{\rmdefault}{\mddefault}{\itdefault}{\color[rgb]{0,0,0}$l_3$}%
}}}}
\put(1432,-6831){\makebox(0,0)[lb]{\smash{{\SetFigFont{12}{14.4}{\rmdefault}{\mddefault}{\itdefault}{\color[rgb]{0,0,0}$l_4$}%
}}}}
\put(6752,-10357){\makebox(0,0)[lb]{\smash{{\SetFigFont{12}{14.4}{\rmdefault}{\mddefault}{\itdefault}{\color[rgb]{0,0,0}$l_1$}%
}}}}
\put(6752,-9648){\makebox(0,0)[lb]{\smash{{\SetFigFont{12}{14.4}{\rmdefault}{\mddefault}{\itdefault}{\color[rgb]{0,0,0}$l_2$}%
}}}}
\put(6752,-8231){\makebox(0,0)[lb]{\smash{{\SetFigFont{12}{14.4}{\rmdefault}{\mddefault}{\itdefault}{\color[rgb]{0,0,0}$l_3$}%
}}}}
\put(6752,-6813){\makebox(0,0)[lb]{\smash{{\SetFigFont{12}{14.4}{\rmdefault}{\mddefault}{\itdefault}{\color[rgb]{0,0,0}$l_4$}%
}}}}
\put(1432,-16525){\makebox(0,0)[lb]{\smash{{\SetFigFont{12}{14.4}{\rmdefault}{\mddefault}{\itdefault}{\color[rgb]{0,0,0}$l_1$}%
}}}}
\put(1432,-15816){\makebox(0,0)[lb]{\smash{{\SetFigFont{12}{14.4}{\rmdefault}{\mddefault}{\itdefault}{\color[rgb]{0,0,0}$l_2$}%
}}}}
\put(1432,-14399){\makebox(0,0)[lb]{\smash{{\SetFigFont{12}{14.4}{\rmdefault}{\mddefault}{\itdefault}{\color[rgb]{0,0,0}$l_3$}%
}}}}
\put(1432,-12981){\makebox(0,0)[lb]{\smash{{\SetFigFont{12}{14.4}{\rmdefault}{\mddefault}{\itdefault}{\color[rgb]{0,0,0}$l_4$}%
}}}}
\end{picture}}
\caption{Obtaining a recurrence for a region of type A when $l_1=1$.}\label{Fig:Onesidehole3}
\end{figure}

Next, we consider a region of Type $A$ with $l_1>1$. We apply Kuo condensation from  Lemma \ref{Kuolem1} with the four vertices $u,v,w,$ and $s$ chosen as in Figure \ref{Fig:Onesidehole4}(b). The locations of these vertices are similar to those in the case where $l_1=1$ above. Kuo condensation implies that the product of TGFs of the two regions in the top row of Figure \ref{Fig:Onesidehole4} is equal to the product of the TGFs of the two regions in the middle row, plus the product of the TGFs of the two regions in the bottom row. We also consider the removal of forced lozenges as shown in the figure and get the following recurrence:
\begin{align}\label{Arecurrence3}
\M(A_{x,d}((l_i)_{i=1}^{m}))&\M(B_{x,d-2}((l_i-1)_{i=1}^{m-1}))\notag\\
=&\frac{q^{2x+2d-m}+q^{-(2x+2d-m)}}{2}\cdot\M(A_{x,d-1}((l_i)_{i=1}^{m-1}))\M(B_{x,d-1}((l_i-1)_{i=1}^{m}))\notag\\
&+\M(B_{x+1,d-2}((l_i-1)_{i=1}^{m-1}))\M(A_{x-1,d}((l_i)_{i=1}^{m})).
\end{align}

We note that, unlike the case $l_1=1$ treated above, the removal of forced lozenges in the regions corresponding to the graphs $G-\{u,v,w,s\}$, $G-\{w,s\}$, and $G-\{u,s\}$ now give new regions of type $B$. It is straightforward to check that the sum of the $h$- and $x$-parameters of the last five regions in the above recurrence are all strictly less than that of the first one.

\begin{figure}\centering
\setlength{\unitlength}{3947sp}%
\begingroup\makeatletter\ifx\SetFigFont\undefined%
\gdef\SetFigFont#1#2#3#4#5{%
  \reset@font\fontsize{#1}{#2pt}%
  \fontfamily{#3}\fontseries{#4}\fontshape{#5}%
  \selectfont}%
\fi\endgroup%
\resizebox{!}{22cm}{
\begin{picture}(0,0)%
\includegraphics{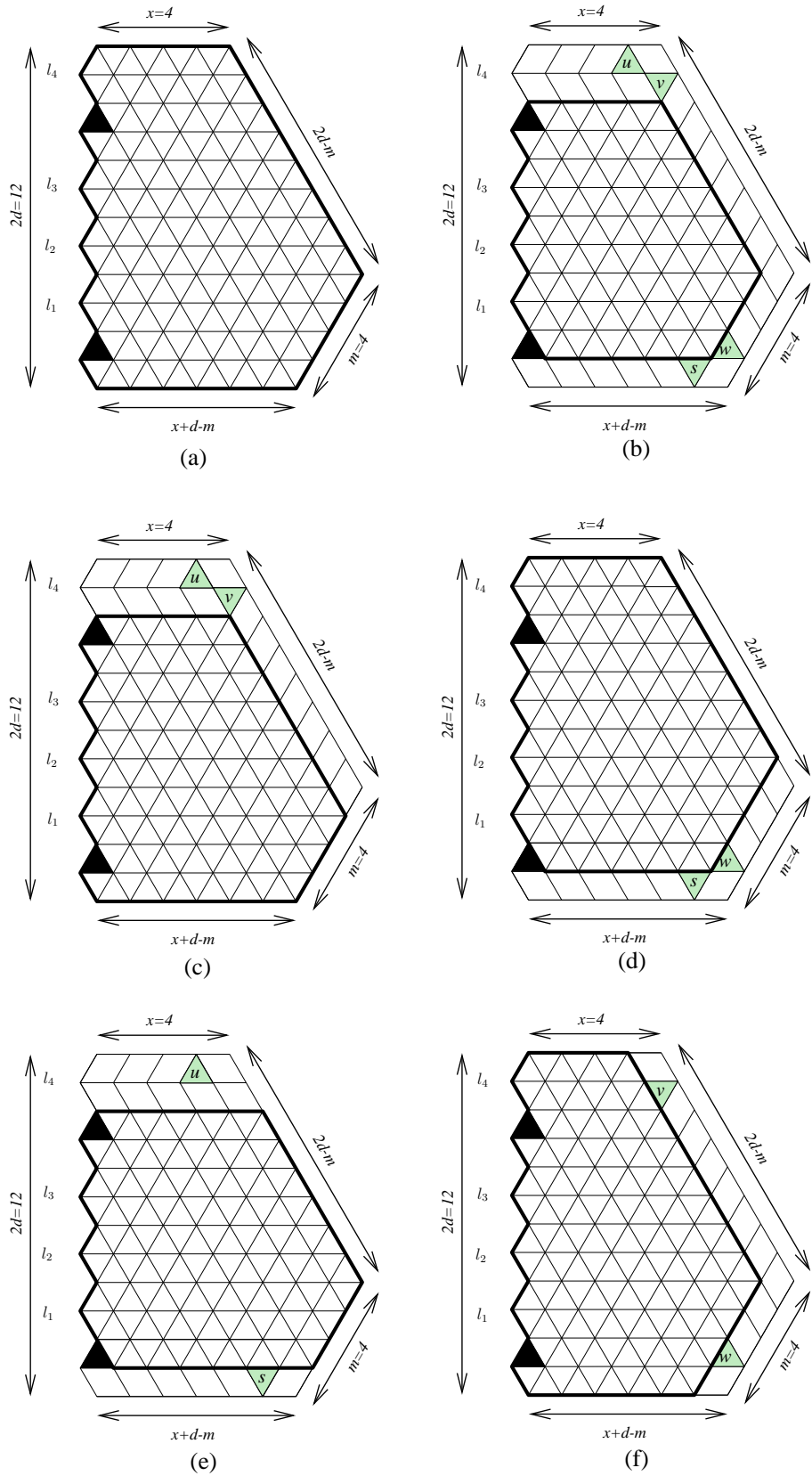}%
\end{picture}%

\begin{picture}(9964,18357)(902,-17877)
\put(6743,-13000){\makebox(0,0)[lb]{\smash{{\SetFigFont{12}{14.4}{\rmdefault}{\mddefault}{\itdefault}{\color[rgb]{0,0,0}$l_4$}%
}}}}
\put(1426,-3376){\makebox(0,0)[lb]{\smash{{\SetFigFont{12}{14.4}{\rmdefault}{\mddefault}{\itdefault}{\color[rgb]{0,0,0}$l_1$}%
}}}}
\put(1411,-2641){\makebox(0,0)[lb]{\smash{{\SetFigFont{12}{14.4}{\rmdefault}{\mddefault}{\itdefault}{\color[rgb]{0,0,0}$l_2$}%
}}}}
\put(1433,-1878){\makebox(0,0)[lb]{\smash{{\SetFigFont{12}{14.4}{\rmdefault}{\mddefault}{\itdefault}{\color[rgb]{0,0,0}$l_3$}%
}}}}
\put(1433,-460){\makebox(0,0)[lb]{\smash{{\SetFigFont{12}{14.4}{\rmdefault}{\mddefault}{\itdefault}{\color[rgb]{0,0,0}$l_4$}%
}}}}
\put(6710,-9777){\makebox(0,0)[lb]{\smash{{\SetFigFont{12}{14.4}{\rmdefault}{\mddefault}{\itdefault}{\color[rgb]{0,0,0}$l_1$}%
}}}}
\put(6695,-9042){\makebox(0,0)[lb]{\smash{{\SetFigFont{12}{14.4}{\rmdefault}{\mddefault}{\itdefault}{\color[rgb]{0,0,0}$l_2$}%
}}}}
\put(6717,-8279){\makebox(0,0)[lb]{\smash{{\SetFigFont{12}{14.4}{\rmdefault}{\mddefault}{\itdefault}{\color[rgb]{0,0,0}$l_3$}%
}}}}
\put(6717,-6861){\makebox(0,0)[lb]{\smash{{\SetFigFont{12}{14.4}{\rmdefault}{\mddefault}{\itdefault}{\color[rgb]{0,0,0}$l_4$}%
}}}}
\put(1385,-15897){\makebox(0,0)[lb]{\smash{{\SetFigFont{12}{14.4}{\rmdefault}{\mddefault}{\itdefault}{\color[rgb]{0,0,0}$l_1$}%
}}}}
\put(1370,-15162){\makebox(0,0)[lb]{\smash{{\SetFigFont{12}{14.4}{\rmdefault}{\mddefault}{\itdefault}{\color[rgb]{0,0,0}$l_2$}%
}}}}
\put(1392,-14399){\makebox(0,0)[lb]{\smash{{\SetFigFont{12}{14.4}{\rmdefault}{\mddefault}{\itdefault}{\color[rgb]{0,0,0}$l_3$}%
}}}}
\put(1392,-12981){\makebox(0,0)[lb]{\smash{{\SetFigFont{12}{14.4}{\rmdefault}{\mddefault}{\itdefault}{\color[rgb]{0,0,0}$l_4$}%
}}}}
\put(6736,-15916){\makebox(0,0)[lb]{\smash{{\SetFigFont{12}{14.4}{\rmdefault}{\mddefault}{\itdefault}{\color[rgb]{0,0,0}$l_1$}%
}}}}
\put(6721,-15181){\makebox(0,0)[lb]{\smash{{\SetFigFont{12}{14.4}{\rmdefault}{\mddefault}{\itdefault}{\color[rgb]{0,0,0}$l_2$}%
}}}}
\put(6743,-14418){\makebox(0,0)[lb]{\smash{{\SetFigFont{12}{14.4}{\rmdefault}{\mddefault}{\itdefault}{\color[rgb]{0,0,0}$l_3$}%
}}}}
\put(6725,-3387){\makebox(0,0)[lb]{\smash{{\SetFigFont{12}{14.4}{\rmdefault}{\mddefault}{\itdefault}{\color[rgb]{0,0,0}$l_1$}%
}}}}
\put(6710,-2652){\makebox(0,0)[lb]{\smash{{\SetFigFont{12}{14.4}{\rmdefault}{\mddefault}{\itdefault}{\color[rgb]{0,0,0}$l_2$}%
}}}}
\put(6732,-1889){\makebox(0,0)[lb]{\smash{{\SetFigFont{12}{14.4}{\rmdefault}{\mddefault}{\itdefault}{\color[rgb]{0,0,0}$l_3$}%
}}}}
\put(6732,-471){\makebox(0,0)[lb]{\smash{{\SetFigFont{12}{14.4}{\rmdefault}{\mddefault}{\itdefault}{\color[rgb]{0,0,0}$l_4$}%
}}}}
\put(1445,-9762){\makebox(0,0)[lb]{\smash{{\SetFigFont{12}{14.4}{\rmdefault}{\mddefault}{\itdefault}{\color[rgb]{0,0,0}$l_1$}%
}}}}
\put(1430,-9027){\makebox(0,0)[lb]{\smash{{\SetFigFont{12}{14.4}{\rmdefault}{\mddefault}{\itdefault}{\color[rgb]{0,0,0}$l_2$}%
}}}}
\put(1452,-8264){\makebox(0,0)[lb]{\smash{{\SetFigFont{12}{14.4}{\rmdefault}{\mddefault}{\itdefault}{\color[rgb]{0,0,0}$l_3$}%
}}}}
\put(1452,-6846){\makebox(0,0)[lb]{\smash{{\SetFigFont{12}{14.4}{\rmdefault}{\mddefault}{\itdefault}{\color[rgb]{0,0,0}$l_4$}%
}}}}
\end{picture}}
\caption{Obtaining a recurrence for a region of type A when $l_1>1$.}\label{Fig:Onesidehole4}
\end{figure}

We now apply Kuo condensation to the dual graph $G$ of the  region $B_{x,d}((l_i)_{i=1}^{m})$ as shown in Figure \ref{Fig:Onesidehole5}. We have the recurrence
\begin{align}\label{Arecurrence4}
\M(B_{x,d}((l_i)_{i=1}^{m}))&\M(A_{x,d-1}((l_i)_{i=1}^{m-1}))\notag\\
=&\frac{q^{2x+2d-m+1}+q^{-(2x+2d-m+1)}}{2}\M(B_{x,d-1}((l_i)_{i=1}^{m-1}))\M(A_{x,d}((l_i)_{i=1}^{m}))\notag\\
&+\M(A_{x+1,d-1}((l_i)_{i=1}^{m-1}))\M(B_{x-1,d}((l_i)_{i=1}^{m})).
\end{align}
Again, the sum of the $h$- and $x$-parameters of the last five regions in the above recurrence are all strictly less than that of the first one.

\begin{figure}\centering
\setlength{\unitlength}{3947sp}%
\begingroup\makeatletter\ifx\SetFigFont\undefined%
\gdef\SetFigFont#1#2#3#4#5{%
  \reset@font\fontsize{#1}{#2pt}%
  \fontfamily{#3}\fontseries{#4}\fontshape{#5}%
  \selectfont}%
\fi\endgroup%
\resizebox{!}{22cm}{
\begin{picture}(0,0)%
\includegraphics{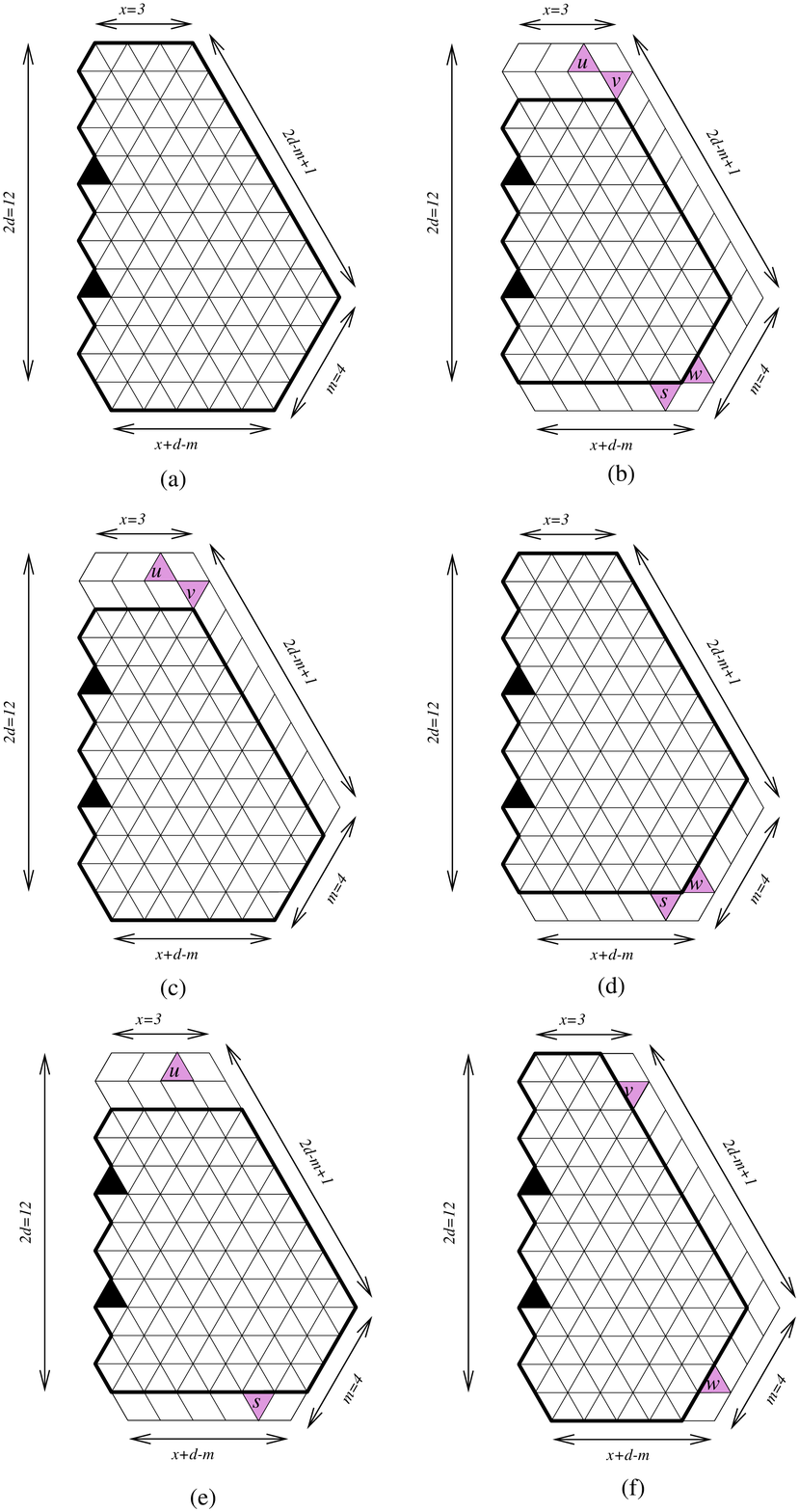}%
\end{picture}%
\begin{picture}(10050,19006)(1023,-18174)
\put(6946,-129){\makebox(0,0)[lb]{\smash{{\SetFigFont{12}{14.4}{\rmdefault}{\mddefault}{\itdefault}{\color[rgb]{0,0,0}$l_4$}%
}}}}
\put(6950,-819){\makebox(0,0)[lb]{\smash{{\SetFigFont{12}{14.4}{\rmdefault}{\mddefault}{\itdefault}{\color[rgb]{0,0,0}$l_3$}%
}}}}
\put(6943,-2229){\makebox(0,0)[lb]{\smash{{\SetFigFont{12}{14.4}{\rmdefault}{\mddefault}{\itdefault}{\color[rgb]{0,0,0}$l_2$}%
}}}}
\put(6950,-3647){\makebox(0,0)[lb]{\smash{{\SetFigFont{12}{14.4}{\rmdefault}{\mddefault}{\itdefault}{\color[rgb]{0,0,0}$l_1$}%
}}}}
\put(1629,-122){\makebox(0,0)[lb]{\smash{{\SetFigFont{12}{14.4}{\rmdefault}{\mddefault}{\itdefault}{\color[rgb]{0,0,0}$l_4$}%
}}}}
\put(1633,-812){\makebox(0,0)[lb]{\smash{{\SetFigFont{12}{14.4}{\rmdefault}{\mddefault}{\itdefault}{\color[rgb]{0,0,0}$l_3$}%
}}}}
\put(1626,-2222){\makebox(0,0)[lb]{\smash{{\SetFigFont{12}{14.4}{\rmdefault}{\mddefault}{\itdefault}{\color[rgb]{0,0,0}$l_2$}%
}}}}
\put(1633,-3640){\makebox(0,0)[lb]{\smash{{\SetFigFont{12}{14.4}{\rmdefault}{\mddefault}{\itdefault}{\color[rgb]{0,0,0}$l_1$}%
}}}}
\put(6952,-6521){\makebox(0,0)[lb]{\smash{{\SetFigFont{12}{14.4}{\rmdefault}{\mddefault}{\itdefault}{\color[rgb]{0,0,0}$l_4$}%
}}}}
\put(6956,-7211){\makebox(0,0)[lb]{\smash{{\SetFigFont{12}{14.4}{\rmdefault}{\mddefault}{\itdefault}{\color[rgb]{0,0,0}$l_3$}%
}}}}
\put(6949,-8621){\makebox(0,0)[lb]{\smash{{\SetFigFont{12}{14.4}{\rmdefault}{\mddefault}{\itdefault}{\color[rgb]{0,0,0}$l_2$}%
}}}}
\put(6956,-10039){\makebox(0,0)[lb]{\smash{{\SetFigFont{12}{14.4}{\rmdefault}{\mddefault}{\itdefault}{\color[rgb]{0,0,0}$l_1$}%
}}}}
\put(1635,-6514){\makebox(0,0)[lb]{\smash{{\SetFigFont{12}{14.4}{\rmdefault}{\mddefault}{\itdefault}{\color[rgb]{0,0,0}$l_4$}%
}}}}
\put(1639,-7204){\makebox(0,0)[lb]{\smash{{\SetFigFont{12}{14.4}{\rmdefault}{\mddefault}{\itdefault}{\color[rgb]{0,0,0}$l_3$}%
}}}}
\put(1632,-8614){\makebox(0,0)[lb]{\smash{{\SetFigFont{12}{14.4}{\rmdefault}{\mddefault}{\itdefault}{\color[rgb]{0,0,0}$l_2$}%
}}}}
\put(1639,-10032){\makebox(0,0)[lb]{\smash{{\SetFigFont{12}{14.4}{\rmdefault}{\mddefault}{\itdefault}{\color[rgb]{0,0,0}$l_1$}%
}}}}
\put(7154,-12791){\makebox(0,0)[lb]{\smash{{\SetFigFont{12}{14.4}{\rmdefault}{\mddefault}{\itdefault}{\color[rgb]{0,0,0}$l_4$}%
}}}}
\put(7158,-13481){\makebox(0,0)[lb]{\smash{{\SetFigFont{12}{14.4}{\rmdefault}{\mddefault}{\itdefault}{\color[rgb]{0,0,0}$l_3$}%
}}}}
\put(7151,-14891){\makebox(0,0)[lb]{\smash{{\SetFigFont{12}{14.4}{\rmdefault}{\mddefault}{\itdefault}{\color[rgb]{0,0,0}$l_2$}%
}}}}
\put(7158,-16309){\makebox(0,0)[lb]{\smash{{\SetFigFont{12}{14.4}{\rmdefault}{\mddefault}{\itdefault}{\color[rgb]{0,0,0}$l_1$}%
}}}}
\put(1837,-12784){\makebox(0,0)[lb]{\smash{{\SetFigFont{12}{14.4}{\rmdefault}{\mddefault}{\itdefault}{\color[rgb]{0,0,0}$l_4$}%
}}}}
\put(1841,-13474){\makebox(0,0)[lb]{\smash{{\SetFigFont{12}{14.4}{\rmdefault}{\mddefault}{\itdefault}{\color[rgb]{0,0,0}$l_3$}%
}}}}
\put(1834,-14884){\makebox(0,0)[lb]{\smash{{\SetFigFont{12}{14.4}{\rmdefault}{\mddefault}{\itdefault}{\color[rgb]{0,0,0}$l_2$}%
}}}}
\put(1841,-16302){\makebox(0,0)[lb]{\smash{{\SetFigFont{12}{14.4}{\rmdefault}{\mddefault}{\itdefault}{\color[rgb]{0,0,0}$l_1$}%
}}}}
\end{picture}}
\caption{Obtaining a recurrence for region of type B.}\label{Fig:Onesidehole5}
\end{figure}

Denote by $f_{x,d}(\textbf{l})$ and $g_{x,y}(\textbf{l})$ the expressions on the right-hand sides of (\ref{Aformula}) and (\ref{Bformula}), respectively. To finish the proof we will verify that  $f_{x,d}(\textbf{l})$ and $g_{x,y}(\textbf{l})$ satisfy recurrences (\ref{Arecurrence2}), (\ref{Arecurrence3}), and (\ref{Arecurrence4}). Equivalently, we need to verify the following three identities (the first two for Theorem \ref{typeAthm} and the third for Theorem \ref{typeBthm}):
\begin{align}\label{Arecurrence5}
\left(\frac{q^{2x+2d-m}+q^{-(2x+2d-m)}}{2}\right)\cdot&\frac{f_{x,d-1}((l_i)_{i=1}^{m-1})f_{x,d-1}((l_i-1)_{i=2}^{m})}{f_{x,d}((l_i)_{i=1}^{m})f_{x,d-2}((l_i-1)_{i=2}^{m-1})}\notag\\
&+\frac{f_{x+1,d-2}((l_i-1)_{i=2}^{m-1})f_{x-1,d}((l_i)_{i=1}^{m})}{f_{x,d}((l_i)_{i=1}^{m})f_{x,d-2}((l_i-1)_{i=2}^{m-1})}=1
\end{align}
when $l_1=1$ and $l_m=d$;
\begin{align}\label{Arecurrence6}
\left(\frac{q^{2x+2d-m}+q^{-(2x+2d-m)}}{2}\right)\cdot&\frac{f_{x,d-1}((l_i)_{i=1}^{m-1})g_{x,d-1}((l_i-1)_{i=1}^{m})}{f_{x,d}((l_i)_{i=1}^{m})g_{x,d-2}((l_i-1)_{i=1}^{m-1})}\notag\\
&+\frac{g_{x+1,d-2}((l_i-1)_{i=1}^{m-1})f_{x-1,d}((l_i)_{i=1}^{m})}{f_{x,d}((l_i)_{i=1}^{m})g_{x,d-2}((l_i-1)_{i=1}^{m-1})}=1
\end{align}
when $l_1>1$ and $l_m=d$;
\begin{align}\label{Arecurrence7}
\left(\frac{q^{2x+2d-m+1}+q^{-(2x+2d-m+1)}}{2}\right)\cdot &\frac{g_{x,d-1}((l_i)_{i=1}^{m-1})f_{x,d}((l_i)_{i=1}^{m})}{g_{x,d}((l_i)_{i=1}^{m})f_{x,d-1}((l_i)_{i=1}^{m-1})}\notag\\
&+\frac{f_{x+1,d-1}((l_i)_{i=1}^{m-1})g_{x-1,d}((l_i)_{i=1}^{m})}{g_{x,d}((l_i)_{i=1}^{m})f_{x,d-1}((l_i)_{i=1}^{m-1})}=1
\end{align}
when $l_m=d$.

This verification is straightforward, as it requires only basic algebra knowledge, however, it is not obvious. For completeness of the proof, we briefly sketch the verification below. 

We first verify (\ref{Arecurrence5}) with $l_1=1$ and $l_m=d$. Let us simplify the fraction involving four $f$-functions in the first term on the left-hand side.

By definition, the four powers of $2$ reduce to just $2$; the four powers of $q$ reduce to $q^{2x+2d-m}$. The four factors corresponding to the fraction $\frac{\prod_{1\leq i < j \leq m}[2(l_j-l_i)]_{q^2}}{\prod_{i=1}^{m}[2l_i-1]_{q^2}!}$ in the $f$-function
simplify to $[2l_m-1]_{q^2}=[2d-1]_{q^2}$. 

Let us consider the contribution of the four factors corresponding to the double-product in the second row of the right-hand side of (\ref{Aformula}). Most of the factors cancel out, and what remains is 
\[\frac{[2x+2d-m]_{q^2}}{[2x+2d-1]_{q^2}[2x+2d]_{q^2}}.\]
Finally, the factors corresponding to the double product  in the last row of the right-hand side of (\ref{Aformula}) simplify
\[\frac{[2(x+d)]_{q^2}}{[2(x+d-m+l_m)]_{q^2}}.\]
The first term on the  the left-hand side of (\ref{Arecurrence5}) now reduces to
\begin{equation}
(q^{4x+4d-2m}+1)\frac{[2d-1]_{q^2}[2x+2d-m]_{q^2}}{[2x+2d-1]_{q^2}[2(x+2d-m)]_{q^2}}.
\end{equation}

Similarly, one can simplify the  second term to
\begin{equation}
(q^{2d-1})\frac{[2x+2d-2m+1]_{q^2}[2x]_{q^2}}{[2x+2d-1]_{q^2}[2(x+2d-m)]_{q^2}}.
\end{equation}
The identity  (\ref{Arecurrence5}) is now equivalent to
\begin{equation}
(q^{4x+4d-2m}+1)\frac{[2d-1]_{q^2}[2x+2d-m]_{q^2}}{[2x+2d-1]_{q^2}[2(x+2d-m)]_{q^2}}+(q^{4d-2})\frac{[2x+2d-2m+1]_{q^2}[2x]_{q^2}}{[2x+2d-1]_{q^2}[2(x+2d-m)]_{q^2}}=1,
\end{equation}
which can be verified directly.

The identities (\ref{Arecurrence6}) and (\ref{Arecurrence7}) can be confirmed similarly. This finishes the proof.
\end{proof}

The combined proof of Theorems \ref{typeCthm} and \ref{typeDthm} is essentially the same and will be omitted. Strictly speaking, the regions of type C and type D are different from the regions of type A and type B due to the difference in the weighting of vertical lozenges along the vertical side. However, all the recurrences in the combined proof of Theorems \ref{typeCthm} and \ref{typeDthm} still hold, as we have no forced vertical lozenges along the vertical side of our region. Therefore, all arguments in the proof still work well for the case of type C and type D regions.

We now prove Theorem \ref{typeSthm}; the proof of Theorem \ref{typeTthm} is similar and hence omitted.

\begin{proof}[Proof of Theorem \ref{typeSthm}]
We prove (\ref{TypeSformula}) by induction on $x+2u+2d$.

The base cases are the situations when in which one of the parameters $u,d,m,$ or $n$ is equal to zero.

If $m=0$, then after removing forced lozenges, our region becomes a region of type $D$, as shown in Figure \ref{Fig:TwosideBase2}(a). Then (\ref{TypeSformula}) follows from Theorem \ref{typeDthm}.

If $n=0$, then after removed forced lozenges, our region becomes a region of type $A$ (reflected over a horizontal line), as illustrated in Figure \ref{Fig:TwosideBase2}(b). Then (\ref{TypeSformula}) follows from Theorem \ref{typeAthm}.

If $u=0$, then $n=0$, then this case reduces to the case treated above, as $u\geq n$. Similarly, if $d=0$, then $m=0$, and this case reduces to the case treated above, as $d\geq m$.

\begin{figure}\centering
\setlength{\unitlength}{3947sp}%
\begingroup\makeatletter\ifx\SetFigFont\undefined%
\gdef\SetFigFont#1#2#3#4#5{%
  \reset@font\fontsize{#1}{#2pt}%
  \fontfamily{#3}\fontseries{#4}\fontshape{#5}%
  \selectfont}%
\fi\endgroup%
\resizebox{!}{12cm}{
\begin{picture}(0,0)%
\includegraphics{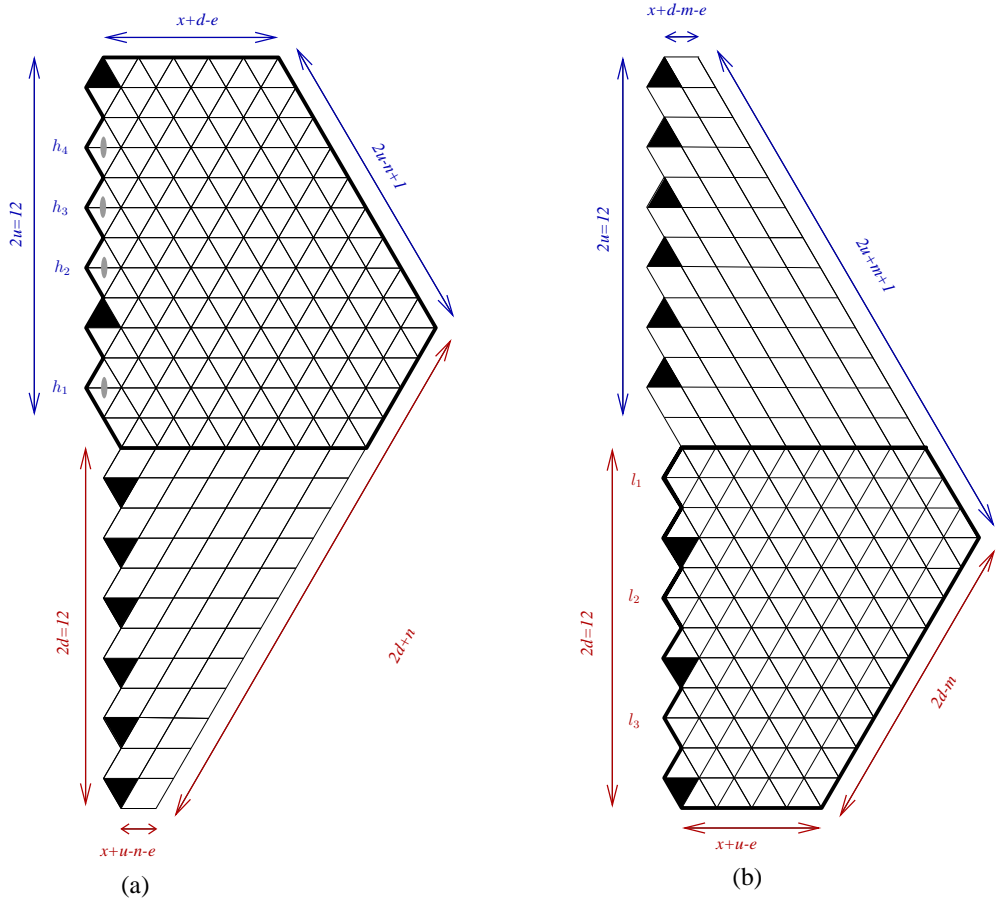}%
\end{picture}%
%
%

\begin{picture}(11629,10674)(1077,-11602)
\put(1651,-5581){\makebox(0,0)[lb]{\smash{{\SetFigFont{12}{14.4}{\rmdefault}{\mddefault}{\itdefault}{\color[rgb]{0,0,.69}$h_1$}%
}}}}
\put(1674,-4171){\makebox(0,0)[lb]{\smash{{\SetFigFont{12}{14.4}{\rmdefault}{\mddefault}{\itdefault}{\color[rgb]{0,0,.69}$h_2$}%
}}}}
\put(1651,-3466){\makebox(0,0)[lb]{\smash{{\SetFigFont{12}{14.4}{\rmdefault}{\mddefault}{\itdefault}{\color[rgb]{0,0,.69}$h_3$}%
}}}}
\put(1651,-2746){\makebox(0,0)[lb]{\smash{{\SetFigFont{12}{14.4}{\rmdefault}{\mddefault}{\itdefault}{\color[rgb]{0,0,.69}$h_4$}%
}}}}
\put(8420,-6678){\makebox(0,0)[lb]{\smash{{\SetFigFont{12}{14.4}{\rmdefault}{\mddefault}{\itdefault}{\color[rgb]{.69,0,0}$l_1$}%
}}}}
\put(8390,-8050){\makebox(0,0)[lb]{\smash{{\SetFigFont{12}{14.4}{\rmdefault}{\mddefault}{\itdefault}{\color[rgb]{.69,0,0}$l_2$}%
}}}}
\put(8390,-9513){\makebox(0,0)[lb]{\smash{{\SetFigFont{12}{14.4}{\rmdefault}{\mddefault}{\itdefault}{\color[rgb]{.69,0,0}$l_3$}%
}}}}
\end{picture}
}
\caption{The cases in which (a) $m=0$ and (b) $n=0$ in the proof of Theorem \ref{typeSthm}.} \label{Fig:TwosideBase2}
\end{figure}

For the induction step, we assume that $u,d,m,$ and $n$ are all positive and that the theorem holds for any $S$-type regions in which the sum of the $x$-parameter, twice the $u$-parameter, and twice the $d$-parameter is strictly less than $x+2u+2d$. We need to show that (\ref{TypeSformula}) holds for any region $S_{x,u,d}((l_i)_{i=1}^{m};(h_j)_{j=1}^n)$. 

We first observe that (\ref{TypeSformula}) holds if $x=0$. Indeed, if $x=0$, then at least one of the north or the south side-lengths of the region is equal to $0$. In this case, we can remove forced lozenges to obtain a smaller $S$-type region as in Figures \ref{Fig:TwosideBase}(a) for the case in which $d-m>u-n$, (b) for the case in which $d-m<u-n$, and (c) for the case in which $d-m=u-n>0$. Hence (\ref{TypeSformula}) follows from the induction hypothesis.
In case where $d-m=u-n=0$, our region has only one tiling  consisting of all vertical lozenges (see Figure \ref{Fig:TwosideBase}(d)). It is easy to  calculate the weight of this tiling and verify (\ref{TypeSformula}).

\begin{figure}\centering
\includegraphics[width=12cm]{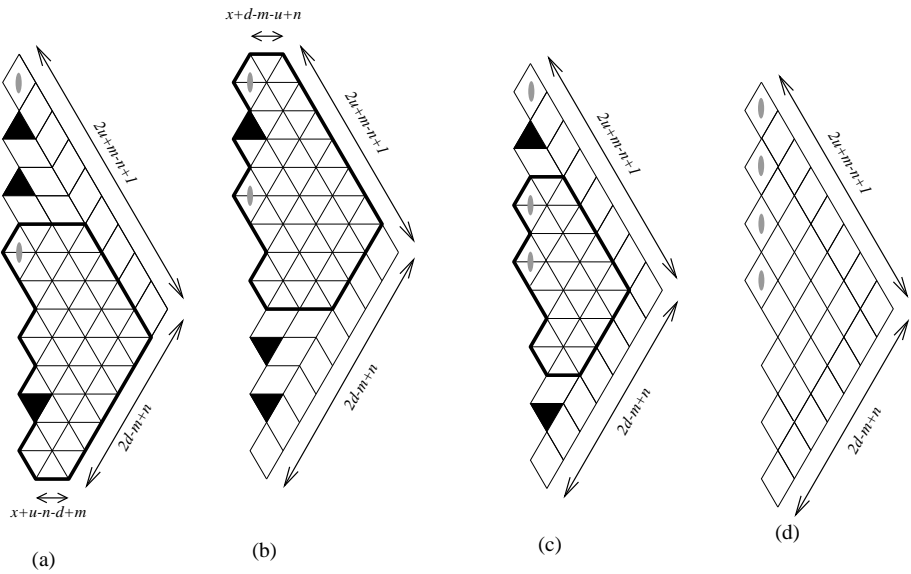}
\caption{Obtaining a smaller $S$-type region when $x=0$.}\label{Fig:TwosideBase}
\end{figure}

If $h_n<u$ or $l_m<d$, then after removing forced lozenges on the top or bottom of the regions, we get the region $S_{x+u-h_n+d-l_m,h_n,l_m}((l_i)_{i=1}^{m};(h_j)_{j=1}^n)$ (see Figure \ref{Fig:Twosidehole2}(a)). Thus
\begin{equation}
\M(S_{x,u,d}((l_i)_{i=1}^{m};(h_j)_{j=1}^n))=\M(S_{x+u-h_n+d-l_m,h_n,l_m}((l_i)_{i=1}^{m};(h_j)_{j=1}^n)),
\end{equation}
and (\ref{TypeSformula}) follows from the induction hypothesis. 

\begin{figure}\centering
\setlength{\unitlength}{3947sp}%
\begingroup\makeatletter\ifx\SetFigFont\undefined%
\gdef\SetFigFont#1#2#3#4#5{%
  \reset@font\fontsize{#1}{#2pt}%
  \fontfamily{#3}\fontseries{#4}\fontshape{#5}%
  \selectfont}%
\fi\endgroup%
\resizebox{!}{12cm}{
\begin{picture}(0,0)%
\includegraphics{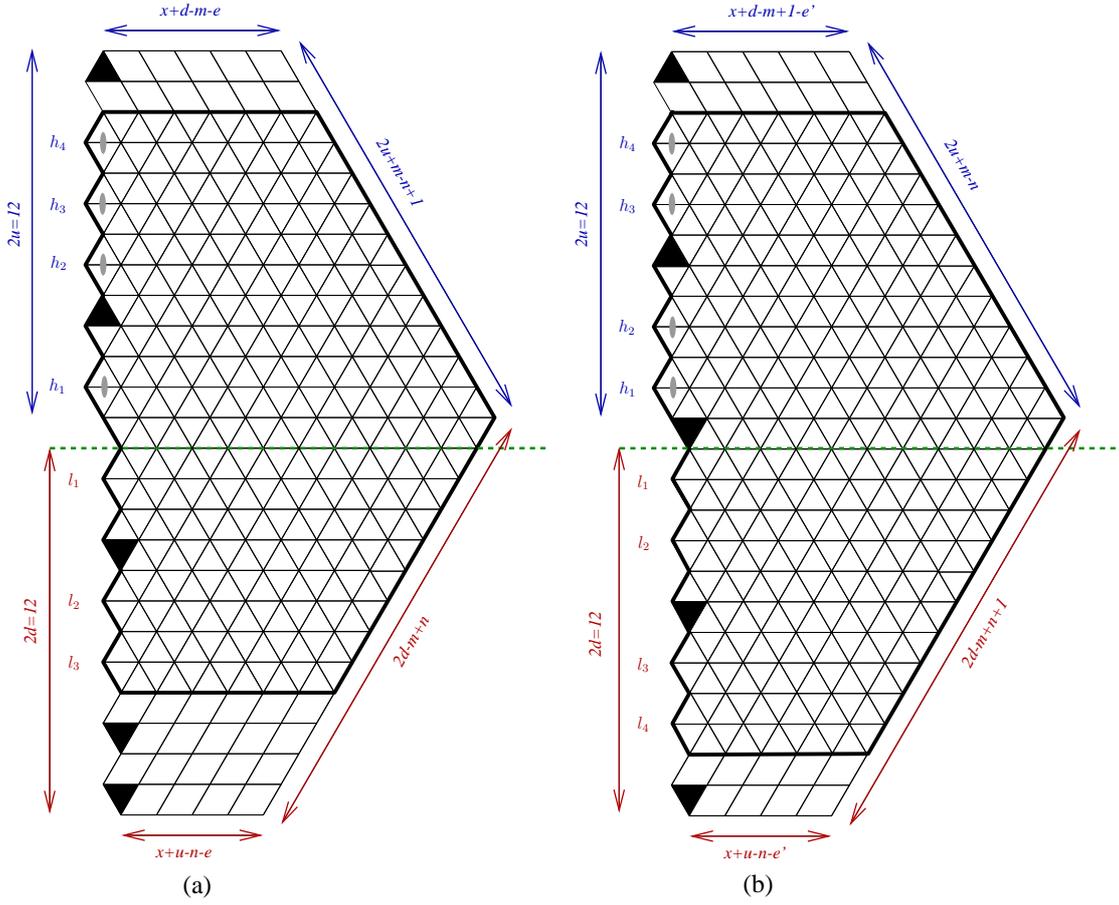}%
\end{picture}%
%
%

\begin{picture}(12876,10492)(1077,-11510)
\put(8394,-9504){\makebox(0,0)[lb]{\smash{{\SetFigFont{12}{14.4}{\rmdefault}{\mddefault}{\itdefault}{\color[rgb]{.69,0,0}$l_4$}%
}}}}
\put(1846,-6676){\makebox(0,0)[lb]{\smash{{\SetFigFont{12}{14.4}{\rmdefault}{\mddefault}{\itdefault}{\color[rgb]{.69,0,0}$l_1$}%
}}}}
\put(1846,-8079){\makebox(0,0)[lb]{\smash{{\SetFigFont{12}{14.4}{\rmdefault}{\mddefault}{\itdefault}{\color[rgb]{.69,0,0}$l_2$}%
}}}}
\put(1846,-8791){\makebox(0,0)[lb]{\smash{{\SetFigFont{12}{14.4}{\rmdefault}{\mddefault}{\itdefault}{\color[rgb]{.69,0,0}$l_3$}%
}}}}
\put(1629,-5581){\makebox(0,0)[lb]{\smash{{\SetFigFont{12}{14.4}{\rmdefault}{\mddefault}{\itdefault}{\color[rgb]{0,0,.69}$h_1$}%
}}}}
\put(1644,-4156){\makebox(0,0)[lb]{\smash{{\SetFigFont{12}{14.4}{\rmdefault}{\mddefault}{\itdefault}{\color[rgb]{0,0,.69}$h_2$}%
}}}}
\put(1636,-3489){\makebox(0,0)[lb]{\smash{{\SetFigFont{12}{14.4}{\rmdefault}{\mddefault}{\itdefault}{\color[rgb]{0,0,.69}$h_3$}%
}}}}
\put(1636,-2761){\makebox(0,0)[lb]{\smash{{\SetFigFont{12}{14.4}{\rmdefault}{\mddefault}{\itdefault}{\color[rgb]{0,0,.69}$h_4$}%
}}}}
\put(8401,-7404){\makebox(0,0)[lb]{\smash{{\SetFigFont{12}{14.4}{\rmdefault}{\mddefault}{\itdefault}{\color[rgb]{.69,0,0}$l_2$}%
}}}}
\put(8394,-8814){\makebox(0,0)[lb]{\smash{{\SetFigFont{12}{14.4}{\rmdefault}{\mddefault}{\itdefault}{\color[rgb]{.69,0,0}$l_3$}%
}}}}
\put(8184,-5619){\makebox(0,0)[lb]{\smash{{\SetFigFont{12}{14.4}{\rmdefault}{\mddefault}{\itdefault}{\color[rgb]{0,0,.69}$h_1$}%
}}}}
\put(8176,-4899){\makebox(0,0)[lb]{\smash{{\SetFigFont{12}{14.4}{\rmdefault}{\mddefault}{\itdefault}{\color[rgb]{0,0,.69}$h_2$}%
}}}}
\put(8184,-3496){\makebox(0,0)[lb]{\smash{{\SetFigFont{12}{14.4}{\rmdefault}{\mddefault}{\itdefault}{\color[rgb]{0,0,.69}$h_3$}%
}}}}
\put(8184,-2769){\makebox(0,0)[lb]{\smash{{\SetFigFont{12}{14.4}{\rmdefault}{\mddefault}{\itdefault}{\color[rgb]{0,0,.69}$h_4$}%
}}}}
\put(8394,-6699){\makebox(0,0)[lb]{\smash{{\SetFigFont{12}{14.4}{\rmdefault}{\mddefault}{\itdefault}{\color[rgb]{.69,0,0}$l_1$}%
}}}}
\end{picture}}
\caption{Obtaining a ``smaller" region in the cases where $h_n<u$ or $l_m<d$ by removing forced lozenges.}\label{Fig:Twosidehole2}
\end{figure}

For the rest of the proof, we assume that $x>0$, $u=h_n$, and $d=l_m$. We now apply Kuo condensation from Lemma \ref{Kuolem1} and obtain the regions shown in Figure \ref{Fig:Twosidehole3}. In particular, the product of the TGFs of the two regions in the top row of Figure \ref{Fig:Twosidehole3} is equal to the product of the TGFs of the two regions in the middle row, plus the product of the TGFs of the two regions in the bottom row. Investigating the removal of forced lozenges yielded by the removal of the $u$-, $v$-, $w$-, $s$-triangles, we obtain from the figure the following recurrence:

\begin{figure}\centering
\includegraphics[width=9cm]{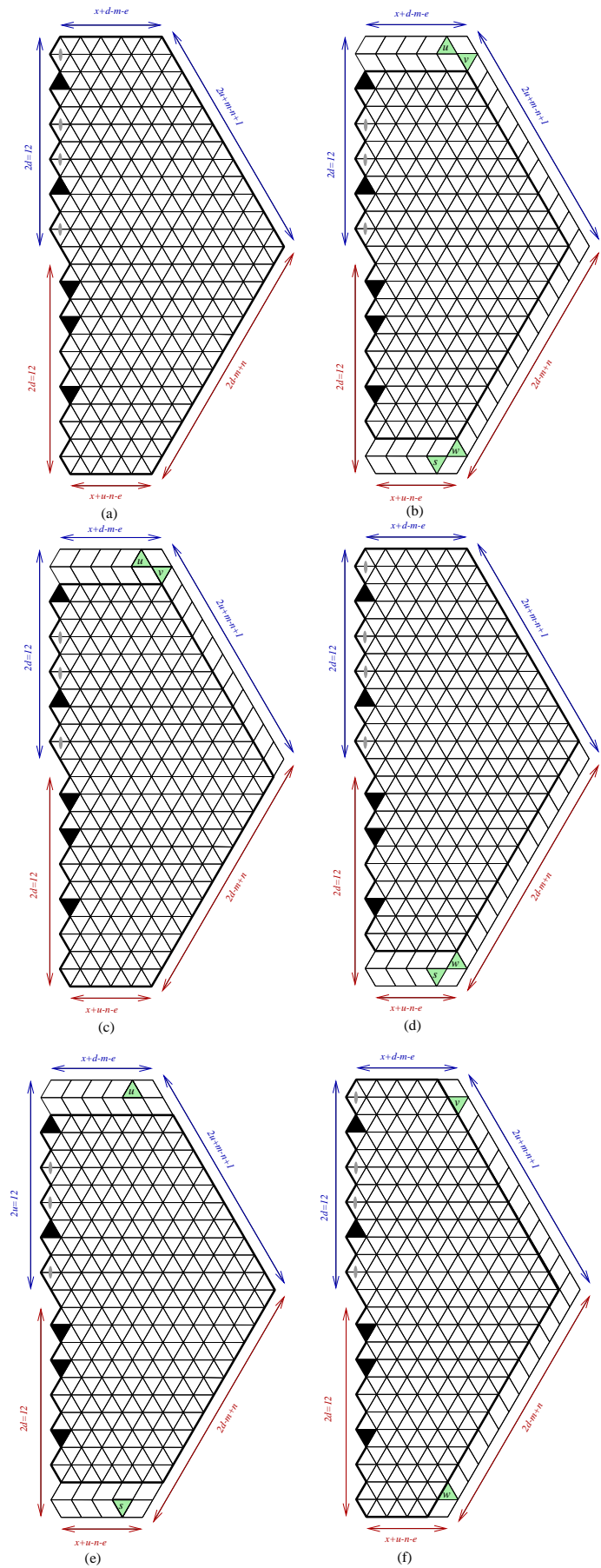}%
\caption{Obtaining a recurrence for the tiling generating function of regions of type S.}\label{Fig:Twosidehole3}
\end{figure}

\begin{align}\label{Srecurrence1}
\M(S_{x,u,d}((l_i)_{i=1}^{m};(h_j)_{j=1}^n))&\M(S_{x,u-1,d-1}((l_i)_{i=1}^{m-1};(h_j)_{j=1}^{n-1}))\notag\\
&=\wt(l_0) \cdot \M(S_{x,u-1,d}((l_i)_{i=1}^{m};(h_j)_{j=1}^{n-1}))\M(S_{x,u,d-1}((l_i)_{i=1}^{m-1};(h_j)_{j=1}^{n}))\notag\\
&+\M(S_{x+1,u-1,d-1}((l_i)_{i=1}^{m-1};(h_j)_{j=1}^{n-1}))\M(S_{x-1,u,d}((l_i)_{i=1}^{m};(h_j)_{j=1}^{n})),
\end{align}
where $\wt(l_0)$ is the weight of the right-most vertical lozenge of the region. It is straightforward to see that
\begin{equation}
\wt(l_0)=\frac{q^{(2x+2u+2d-2e-m-n)}+q^{-(2x+2u+2d-2e-m-n)}}{2}.
\end{equation}
Therefore the above recurrence becomes

\begin{align}\label{Srecurrence2}
\M(S_{x,u,d}&((l_i)_{i=1}^{m};(h_j)_{j=1}^n))\M(S_{x,u-1,d-1}((l_i)_{i=1}^{m-1};(h_j)_{j=1}^{n-1}))\notag\\
&=\frac{q^{(2x+2u+2d-2e-m-n)}+q^{-(2x+2u+2d-2e-m-n)}}{2} \notag\\
&\times\M(S_{x,u-1,d}((l_i)_{i=1}^{m};(h_j)_{j=1}^{n-1}))\M(S_{x,u,d-1}((l_i)_{i=1}^{m-1};(h_j)_{j=1}^{n}))\notag\\
&+\M(S_{x+1,u-1,d-1}((l_i)_{i=1}^{m-1};(h_j)_{j=1}^{n-1}))\M(S_{x-1,u,d}((l_i)_{i=1}^{m};(h_j)_{j=1}^{n})).
\end{align}

Our remaining job is to verify that the expression on the right-hand side of (\ref{TypeSformula}) satisfies the same recurrence. Equivalently, we need to show
\begin{align}\label{Srecurrence3}
&\left(\frac{q^{(2x+2u+2d-2e-m-n)}+q^{-(2x+2u+2d-2e-m-n)}}{2}\right) \frac{P_{x,u-1,d}((l_i)_{i=1}^{m};(h_j)_{j=1}^{n-1})P_{x,u,d-1}((l_i)_{i=1}^{m-1};(h_j)_{j=1}^{n})}{P_{x,u,d}((l_i)_{i=1}^{m};(h_j)_{j=1}^n)P_{x,u-1,d-1}((l_i)_{i=1}^{m-1};(h_j)_{j=1}^{n-1})}\notag\\
&+\frac{P_{x+1,u-1,d-1}((l_i)_{i=1}^{m-1};(h_j)_{j=1}^{n-1})P_{x-1,u,d}((l_i)_{i=1}^{m};(h_j)_{j=1}^{n})}{P_{x,u,d}((l_i)_{i=1}^{m};(h_j)_{j=1}^n)P_{x,u-1,d-1}((l_i)_{i=1}^{m-1};(h_j)_{j=1}^{n-1})}=1.
\end{align}
Again, the verification of this identity is straightforward, but not trivial. For the completeness, we briefly show the process below.

Let us consider the fraction involving the four $P$-polynomials in the first term on the left-hand side. Using the definition of the $P$-polynomial, one can simplify this fraction to
\begin{align}
2&q^{(2x+2u+2d-2e-m-n)}[2d+2u]_{q^2}\frac{[2x+2d+2u-2e-2]_{q^2}}{[2x+2d+2u-2e]_{q^2}}\notag\\
&\times \frac{[2x+2d+2u-2e-m-n]_{q^2}}{[2x+2d+2u-2e-1]_{q^2}[2x+2d+2u-2e]_{q^2}}\notag\\
&\times \frac{[2x+2d+2u-2e-1]_{q^2}}{[2x+2d+2u-2e-2]_{q^2}}\frac{[2x+2d+2u-2e]_{q^2}}{[2x+4d+2u-2e-2m]_{q^2}[2x+2d+4u-2e-2n]_{q^2}}\notag\\
&=\frac{2q^{(2x+2u+2d-2e-m-n)}[2d+2u]_{q^2}[2x+2d+2u-2e-m-n]_{q^2}}{[2x+4d+2u-2e-2m]_{q^2}[2x+2d+4u-2e-2n]_{q^2}}.
\end{align}
Then the first term of (\ref{Srecurrence3}) now becomes
\begin{align}
\frac{(q^{(4x+4u+4d-4e-2m-2n)}+1)[2d+2u]_{q^2}[2x+2d+2u-2e-m-n]_{q^2}}{[2x+4d+2u-2e-2m]_{q^2}[2x+2d+4u-2e-2n]_{q^2}}.
\end{align}
Similarly, one could simplify the second term on the left-hand side of (\ref{Srecurrence3}) to
\begin{align}
\frac{q^{4d+4u}[2x+2u-2e-2n]_{q^2}[2x+2d-2e-2m]_{q^2}}{[2x+4d+2u-2e-2m]_{q^2}[2x+2d+4u-2e-2n]_{q^2}}.
\end{align}
Then (\ref{Srecurrence3}) is equivalent to 
\begin{align}
&\frac{(q^{(4x+4u+4d-4e-2m-2n)}+1)[2d+2u]_{q^2}[2x+2d+2u-2e-m-n]_{q^2}}{[2x+4d+2u-2e-2m]_{q^2}[2x+2d+4u-2e-2n]_{q^2}}\notag\\
&+\frac{q^{4d+4u}[2x+2u-2e-2n]_{q^2}[2x+2d-2e-2m]_{q^2}}{[2x+4d+2u-2e-2m]_{q^2}[2x+2d+4u-2e-2n]_{q^2}}=1,
\end{align}
which is a true identity. This finishes the proof.
\end{proof}

\bibliographystyle{plain}
\bibliography{WHHedits}

\end{document}